\documentclass[10pt]{article}

\usepackage{amsfonts,a4wide,amsmath,amssymb,amsthm}
\usepackage[utf8]{inputenc}  
\usepackage[T1]{fontenc}
\usepackage[all]{xy}    
\usepackage{mathtools}  
\usepackage{multirow}    

\usepackage{hyperref}      
\hypersetup{
     colorlinks=true, 
     breaklinks=true, 
     urlcolor= black,  
     linkcolor= black, 
     linktoc	=all,	
     citecolor=black,	
     filecolor=black,	
     bookmarksopen=true,            
     pdftitle={\texorpdfstring{Surjectivity of Galois representations associated with quadratic $\mathbb{Q}$-curves}{Surjectivity of Galois representations associated with quadratic Q-curves}}, 
     pdfauthor={Samuel Le Fourn},     
     pdfsubject={Représentations galoisiennes}          
}
\usepackage{tikz}   

\theoremstyle{plain}
\newtheorem{thm}{Theorem}[section]
\newtheorem{prop}[thm]{Proposition}
\newtheorem{lem}[thm]{Lemma}
\newtheorem*{cor}{Corollary}
\newtheorem*{thmsansnom}{Theorem}

\theoremstyle{definition}
\newtheorem{defi}{Definition}[thm]

\newtheorem*{defisansnom}{Definition}

\theoremstyle{remark}
\newtheorem{rem}{Remark}[thm]
\newtheorem*{remetoile}{Remark}

\newcommand{\g}{\gamma}

\newcommand{\s}{\sigma}
\renewcommand{\t}{\tau}

\newcommand{\G}{\Gamma}

\newcommand{\gA}{{\mathfrak{A}}}

\newcommand{\gP}{{\mathfrak{P}}}

\newcommand{\gS}{{\mathfrak{S}}}

\newcommand{\Acal}{{\mathcal A}}
\newcommand{\Bcal}{{\mathcal B}}
\newcommand{\Ccal}{{\mathcal C}}
\newcommand{\Dcal}{{\mathcal D}}
\newcommand{\Ecal}{{\mathcal E}}
\newcommand{\Fcal}{{\mathcal F}}
\newcommand{\Gcal}{{\mathcal G}}
\newcommand{\Hcal}{{\mathcal H}}
\newcommand{\Ical}{{\mathcal I}}

\newcommand{\Ocal}{{\mathcal O}}
\newcommand{\Pcal}{{\mathcal P}}

\newcommand{\Scal}{{\mathcal S}}

\newcommand{\N}{{\mathbb{N}}}
\newcommand{\Z}{{\mathbb{Z}}}
\newcommand{\Q}{{\mathbb{Q}}}
\newcommand{\R}{{\mathbb{R}}}
\newcommand{\C}{{\mathbb{C}}}
\renewcommand{\P}{\mathbb{P}}
\newcommand{\F}{\mathbb{F}}
\newcommand{\Fp}{{\mathbb{F}_{\! p}}}
\newcommand{\Fl}{{\mathbb{F}_{\! \ell}}}

\newcommand{\T}{{\mathbb{T}}}

\newcommand{\Aut}{\operatorname{Aut}}

\renewcommand{\div}{\operatorname{div}}

\newcommand{\End}{\operatorname{End}}
\newcommand{\Gal}{\operatorname{Gal}}
\newcommand{\GL}{\operatorname{GL}}
\newcommand{\SL}{\operatorname{SL}}
\newcommand{\Id}{\operatorname{Id}}

\newcommand{\Spec}{\operatorname{Spec}}

\newcommand{\cl}{\operatorname{cl}}

\newcommand{\PGL}{\operatorname{PGL}}
\newcommand{\PSL}{\operatorname{PSL}}

\newcommand{\num}{\operatorname{num}}

\newcommand{\im}{\operatorname{im}}

\newcommand{\mt}{\mapsto}	
\newcommand{\lmt}{\longmapsto}
\newcommand{\ra}{\rightarrow}

\newcommand{\Llra}{\Longleftrightarrow}

\newcommand{\fonction}[5]{\begin{array}{c|ccl}           
#1: & #2 & \longrightarrow & #3 \\
    & #4 & \longmapsto & #5 \end{array}}

\newcommand{\quot}[2]                                    
{\raisebox{.6ex}{\newline$#1$}\!/\!\raisebox{-.6ex}{$#2$}}

\newcommand{\glp}{\GL_2 ( \F_p )}
\renewcommand{\glp}{\GL_2 ( \Fp )}
\newcommand{\glep}{\GL ( E_p )}

\newcommand{\pglep}{\PGL ( E_p )}

\newcommand{\slz}{\SL_2(\Z)}

\newcommand{\xsp}{X_{\rm{split}} (p)}
\newcommand{\xnsp}{X_{\rm{nonsplit}}(p)}

\newcommand{\xosdp}{X_0 ^s (d ; p)}
\newcommand{\xonsdp}{X_0 ^{ns} (d;p)}
\newcommand{\xospcardp}{X_0 ^{\rm{sC}} (d;p)}

\newcommand{\xspcar}{X_{\rm{sp.Car.}}(p)}

\newcommand{\tildeJp}{\widetilde{J} (p)}

\newcommand{\Kb}{\overline{K}}
\newcommand{\Qb}{\overline{\Q}}
\newcommand{\GalK}{\Gal (\Kb / K)}
\newcommand{\GalQ}{{\Gal(\Qb / \Q)}}

\newcommand{\Fpdeux}{{\mathbb{F}_{p ^2}}}
\newcommand{\Fpb}{{\overline{\Fp}}}

\newcommand{\rep}{\rho_{E,p}}
\newcommand{\Prep}{\P \overline{\rho}_{E,p}}
\newcommand{\Prepprime}{\P \overline{\rho}_{E',p}}

\newcommand{\diagstar}{\begin{pmatrix} \ast & 0 \\ 0 & \ast \end{pmatrix}}

\newcommand{\antidiagstar}{\begin{pmatrix} 0 & \ast \\ \ast & 0 \end{pmatrix}}

\newcommand{\matabcd}{\begin{pmatrix} a & b \\ c & d \end{pmatrix}}

\newcommand{\ok}{{\Ocal_K}}

\title{\Huge Surjectivity of Galois representations associated with quadratic $\Q$-curves}
\author{Samuel Le Fourn \footnote{Email : \url {Samuel.Le.Fourn@math.u-bordeaux1.fr}  }\\
Université de Bordeaux I}
\date\today

\begin{document}
\maketitle
\begin{abstract}
We prove in this paper an uniform surjectivity result for Galois representations associated with non-CM $\Q$-curves over imaginary quadratic fields, using various tools for the proof, such as Mazur's method, isogeny theorems,  Runge's method and analytic estimates of sums of $L$-functions.
\smallskip


\end{abstract}

\section*{Introduction}



For every elliptic curve $E$ defined over a number field $K$ and every prime number $p$, the representation
\[
 \rep : \GalK \ra \glep \cong \glp,
\]
defined by the action of $\GalQ$ on the $p$-torsion $E_p$ of $E$, is a central object in the study of elliptic curves. Serre proved in 1972 \cite{Serre71} that for any elliptic curve $E$ without complex multiplication and defined over a number field $K$, the representation $\rep$ is surjective for large enough $p$, the bound depending on $E$ and $K$.  In fact, we prove here as a side result a totally explicit version of Serre's result (Theorem \ref{Serreexplicite}) which might be of independent interest : in particular, it asserts that for any such elliptic curve $E$, the representation $\rep$ is surjective for  
\[
 p > 10^7 [K : \Q] ^2 \left(  \max \{ h_\Fcal (E), 985 \} + 4 \log [K : \Q] \right) ^2
\]
not dividing the discriminant of $K$, where $h_\Fcal (E)$ is the stable Faltings height of $E$.
What is now known as ``Serre's uniformity problem'' is determining whether this bound can be made independent on $E$.

So far, little is known about this problem for general number fields. Over the field $\Q$, Mazur proved in 1977 \cite{Mazur77} that for any elliptic curve $E$ defined over $\Q$ without complex multiplication, the representation $\rep$ is irreducible when $p>37$.  Recently, Bilu, Parent and Rebolledo proved that for such an $E$, the image of $\rep$ is also not contained in the normaliser of a split Cartan subgroup of $\glp$ for $p \geq 11$, $p \neq 13$.

The present work does not deal with the uniformity problem for elliptic curves defined over $\Q$, but for a slightly different object called $\Q$-curve.

\begin{defisansnom}
Let $K$ be a number field. An elliptic curve $E$ defined over $K$ is called a $\Q$-curve if for every  $\s \in \GalQ$, the elliptic curve $E ^\s$ is isogenous to $E$. Its degree $d(E)$ is then the least common multiple of the minimal degrees of isogenies between $E$ and its conjugates.
\end{defisansnom}

\begin{remetoile}
 The set of $\Q$-curves is stable by isogeny. In particular, every elliptic curve defined over $\Qb$ which is isogenous to an elliptic curve defined over $\Q$ is a $\Q$-curve.
To put aside this special case, we will say $E$ is a \emph{strict $\Q$-curve} if it is not isogenous to an elliptic curve defined over $\Q$.
\end{remetoile}

For every $\Q$-curve $E$ of degree $d$ without complex multiplication and every $p$ not dividing $d$, one defines in subsection \ref{rappelsrepproj} a projective representation
\[
\Prep : \GalQ \ra \pglep
\]
analogous to the representation $\rep$ for $\Q$-curves. One small difference with elliptic curves over $\Q$ is that the Weil pairing does not always guarantee the surjectivity of the determinant anymore (but this obstruction to surjectivity only exists when $p$ is ramified in the field of definition of $E$). Therefore, we say $\Prep$ is quasi-surjective if $\Prep$ contains the projection of $\SL (E_p)$ into $\pglep$, to ignore these questions of image of determinant.

 As for modularity, Ribet proved \cite{Ribet04}, using Serre's conjectures (now proved by Khare, Kisin and Wintenberger), that the $\Q$-curves are exactly the modular elliptic curves, that is the elliptic curves that appear as quotients of a modular curve $X_1 (N)$ for some $N$. The ``modular machinery'' therefore works for these curves too, and gives new applications for ternary diophantine equations. In fact, in a certain number of cases, one can attach to a ternary diophantine equation a Frey elliptic curve, as in the celebrated Fermat case, and this curve often happens to be a $\Q$-curve (see, for instance,  \cite{Ellenberg04} or \cite{DieulefaitFreitas}). This has been one of the motivations for the study of such objects in the recent period.

The main result of the present article is the following :

\begin{thmsansnom}
Let $K$ be an imaginary quadratic field of discriminant $-D_K$. 

For every strict $\Q$-curve $E$ defined over $K$ without complex multiplication of degree $d(E)$,  the representation $\Prep$ is surjective for every prime number $p > \max (50 D_K^{1/4} \log(D_K),2 \cdot 10^{13})$ not dividing $D_K d(E)$.
\end{thmsansnom}

We actually give a more precise statement of this theorem in section \ref{theoremesdisogenie}.

Let us insist that this theorem does not add to our knowledge of elliptic curves over $\Q$ (the reader will notice the hypothesis ``strict $\Q$-curve''). In particular, the issues arising for the nonsplit Cartan case over $\Q$ still hold, even if this case is solved in our context.

The theorem is entirely explicit in hope this might apply to diophantine equations, using Frey curves, but but our motivation was to give what seems to be the first instance of a surjectivity theorem for families of elliptic curves. It can even be interpreted as a ``uniform big Galois image'' result for families of abelian surfaces, taking Weil restriction from $K$ to $\Q$ for the $\Q$-curves over quadratic imaginary fields.

The proof mechanism improves on \cite{Ellenberg04} and is akin to that of \cite{BiluParent11}. More precisely, let $K$ be an \emph{imaginary quadratic field}. For the rest of the introduction, $E$ refers to a strict $\Q$-curve without complex multiplication, defined over $K$. Here is the structure of the proof :

{$(\emph{\textbf{0}})$} Use classical knowledge on $\Q$-curves to reduce the problem to a question about rational points on modular curves in subsection \ref{rappelsrepproj}. 

 {$(\emph{\textbf{I}})$} Prove that for large enough $p$ (not depending on $E$), if $\Prep$ is not quasi-surjective, $E$ has potentially good reduction at every prime of $\ok$ (that is, $j(E) \in \ok$). This follows from the classical Mazur's method, designed in \cite{Mazur77}. From Proposition \ref{quasisurjprojrappel} below, this part of the proof splits as follows :
 \begin{equation}
 \label{decoupageI}
 (I) = (I)_{B} + (I)_{SC} + (I)_{NSC} + (I)_{Exc}
 \end{equation}
where $(I)_\ast$ means that for large enough $p$, if $\Prep$ is contained in a maximal subgroup of $\pglep$ of the shape $(\ast)$, then $j(E) \in \ok$, with the four possible types of maximal subgroups of $\pglep$ are denoted by $B$ for Borel subgroup, $SC$ (resp. $NSC$) for the normaliser of a split (resp. nonsplit) Cartan subgroup and $Exc$ for an exceptional subgroup.
   
This step improves Theorem 3.14 of \cite{Ellenberg04}, as with the same hypothesis, we obtain potentially good reduction everywhere for $E$, even for primes above 2 and 3. This improvement is crucial for Runge's method in part $(II)$ below. The four cases are of various difficulty :

$(I)_{Exc}$ does not use Mazur's method, as $\Prep$ is not in the exceptional case for large enough $p$ (see subsection \ref{rappelsrepproj}, using results of \cite{Serre71} on the action of tame inertia).
 
$(I)_B$ et $(I)_{SC}$ will respectively be proved in subsections \ref{Borel case} and \ref{Split Cartan case} with bounds independent of $K$, and even for elliptic curves defined over $\Q$ and $\Q$-curves over $K$ real quadratic. The Borel and split Cartan case use results from subsections \ref{FormimmMazur} and \ref{Eisquotcompongroup}. Our results improve (as explained above) Propositions 3.2 and 3.4 of \cite{Ellenberg04}.

$(I)_{NSC}$ will be stated in subsection \ref{Nonsplit Cartan case} and proved in the Appendix, with a bound depending on $D_K$ and only for strict $\Q$-curves : we only improve quantitatively Proposition 3.6 of \cite{Ellenberg04} here by using different estimates for weighted sums of $L$-functions.

$(\emph{{\textbf{II}}})$ Use Runge's method in section \ref{Runge} to get an \emph{upper bound} of the shape
\[
\log |j(E)|  \leq C \sqrt{d(E)}
\]
for every $\Q$-curve $E$ of degree $d(E)$ defined over an imaginary quadratic field and with an integral $j$-invariant, where $C$ is an absolute explicit constant.

$(\emph{{\textbf{III}}})$ 
From Gaudron-Rémond's version of the period theorem (\cite{GaudronRemond}, Theorem 1.2), and the associated isogeny theorems (giving explicit and sharper versions of Masser-Wüstholz theorem \cite{MasserWustholz93} and Pellarin's theorem \cite{Pellarin01}), we obtain in section \ref{theoremesdisogenie} a new explicit version of Serre's surjectivity theorem, which in turns gives us a \emph{lower bound} of the shape
\[
\log |j(E)| \geq C' \sqrt{d(E) p }
\]
(and even better for Cartan cases) for every $\Q$-curve $E$ of degree $d(E)$, without complex multiplication, defined over a quadratic field, whose $j$-invariant is integral and such that $\Prep$ is not quasi-surjective. Here again, the constant $C'$ is absolute and explicit.

$(\emph{{\textbf{IV}}})$
Finally, gather the previous results to obtain, for any quadratic imaginary field $K$, an explicit bound $p_K$ such that for any prime number $p > p_K$ and any strict $\Q$-curve $E$ of degree $d(E)$ (prime to $p$) over $K$ whose representation $\Prep$ is not quasi-surjective, $j(E) \in \ok$ and 
\[
C' \sqrt{d(E) p } \leq \log |j(E)| \leq C \sqrt{d(E)}
\]
which is impossible for large enough $p$ (independant of $d(E)$).
Therefore, there is a bound $M_K$ such that $\Prep$ is surjective for all $p > M_K$ not dividing $D_K d(E)$, and computing this bound gives us the main the theorem.

The problem of surjectivity for quadratic $\Q$-curves can be asked for $\Q$-curves on larger fields. We expect at least the $(I)_B$, $(I)_{SC}$ and $(I)_{Exc}$ parts to be feasible in the same way for any $\Q$-curve, giving bounds for potentially good reduction depending only on the degree of its field of definition. Moreover, Runge's method in part $(\textbf{$II$})$ demands that for a $\Q$-curve of degree $d$ defined on $K$ with an integral $j$-invariant, there are more cusps on $X_0 (d)$ than there are infinite places on $K$ (hence the ``imaginary quadratic field'' hypothesis), so it might be adaptable to polyquadratic fields that are not totally real.
Finally, part $(\textbf{$III$})$ is very general and gives the same type of bounds, with constants $C$ and $C'$ depending only on the degree of the field of definition of the $\Q$-curve. Consequently, there is some hope for similar results for $\Q$-curves over larger fields, the thorniest issue remaining the nonsplit Cartan case.

\section*{Acknowledgements}
I would like to thank my advisor, Pierre Parent, for his guidance to help me approach the topic, and patience throughout the many readings and corrections of this article. I also am very grateful to Gaël Rémond for his precious help and thorough checks and corrections, especially for the isogeny theorem presented here and its consequences.

\setcounter{tocdepth}{2}
\begin{tableofcontents}
\end{tableofcontents}

\section*{Notations}
In this article, unless stated otherwise, we denote by

{
\begin{tabular}{ll}
$p$ & a prime number larger than $5$. \\
$E$ &  an elliptic curve defined over $\Qb$. \\
$E^\s$ & the Galois conjugate of $E$ by $\s \in \GalQ$.\\
$E_n$ or $E[n]$ & the $n$-torsion of $E$, non-canonically isomorphic to $(\Z/ n \Z) ^2$.\\
$\T$ & the Hecke algebra for $\G_0 (p)$, generated over $\Z$ by the usual Hecke operators $T_n, n \in \N$. \\
$(a,b)$ & the greatest common divisor of the integers $a$ and $b$.
\end{tabular}
}

\noindent For every scheme $X=X_\Z$ over $\Spec \Z$, we denote by :

$X_\Q$ the generic fiber of $X$, considered as a variety over $\Q$.

$X_R : = X \otimes_\Z R$ the extension of scalars from $\Z$ to any ring $R$.

$X_\Fpb : = X \otimes_\Z \Fpb$ the geometric extension of the special fiber of $X$ at $p$.

$\widetilde{X}$ the regular minimal model of $X$ on $\Spec \Z$.

$\widetilde{X_R}$ the regular minimal model of $X$ on $R$ (generally different from $(\widetilde{X})_R$) if $R$ is a Dedekind ring of characteristic 0.
\\

 \noindent For every abelian variety $J=J_\Q$ defined over $\Q$, we denote by :

$J_\Z$ the Néron model of $J$ over $\Z$.

$J_R$ the Néron model of $J$ over any discrete valuation ring $R$ of characteristic 0.

$J_\Fpb : = J_\Z \otimes_\Z \Fpb$ the geometric extension of the fiber of $J_\Z$ at $p$.

$J(\Q)_{\rm{tors}}$ the finite group of rational torsion points of $J$.

\section{\texorpdfstring{Setup of the surjectivity problem and tools for Mazur's method}{Setup of the surjectivity problem and tools for Mazur's method}}
\subsection{\texorpdfstring{$\Q$-curves and moduli spaces associated to the surjectivity problem}{Q-curves and moduli spaces associated to the surjectivity problem}}
\label{rappelsrepproj}

We assume throughout this article that every considered $\Q$-curve is without complex multiplication (which is the natural hypothesis for Serre's surjectivity problem).

\begin{defi}Let $K$ be a number field.
 Let $E$ be a $\Q$-curve \emph{without complex multiplication} defined over $K$.
For every prime number $p$ not dividing $d(E)$, the map
 \[
 \fonction{\Prep}{\GalQ}{\PGL(E_p)}{\s}{\left(D \lmt D_\s := \mu_\s ( D ^\s)\right)}
 \]
for every $\Fp$-line $D$ of $E_p$, is a projective representation of $\GalQ$ in $\P E_p$, which does not depend on the choice of the isogenies $\mu_\s : E^\s \ra E$ of degree prime to $p$. For a fixed embedding $K \subset \Qb$, the restriction of $\Prep$ to $\Gal ( \Qb / K)$ is the projectivization of the natural representation $\rep : \Gal ( \overline{\Q}/K) \ra \GL(E_p)$ on $p$-torsion points of $E$.
\end{defi}

\label{rappelprobsurj}

To put aside the problem of surjectivity of determinant (entirely described by the degrees of the isogenies $\mu_\s : E^\s \ra E$, as the reader can check using Weil pairing), we make the following definition.
\begin{defi}
Let $E$ be a $\Q$-curve and $p$ a prime number not dividing $d(E)$. 
The representation $\Prep$ is quasi-surjective if its image contains $\PSL (E_p)$ (it is then $\PSL (E_p)$ or $\PGL (E_p)$).
\end{defi}

The following proposition (see \cite{Serre71}, $\mathsection$ 2.4 to 2.6) is a consequence of Dickson's theorem on maximal subgroups of $\PGL_2 ( \Fp)$ and justifies the equality \eqref{decoupageI} in the introduction.

\begin{prop}
\label{quasisurjprojrappel}
Let $p$ be a prime number and $K$ be a number field.
Let $E$ be a $\Q$-curve without complex multiplication, defined over $K$ and of degree prime to $p$.
If $\Prep$ is not quasi-surjective, its image is included in one of the four following types of groups : 

$\bullet$ A Borel subgroup of $\pglep$, which means $\Prep$ leaves invariant an $\Fp$-line $C_p$ (Borel case).

 $\bullet$ The normaliser of a split Cartan subgroup of $\pglep$, which means $\Prep$ leaves globally stable a pair $\{A_p,B_p\}$ of distinct $\Fp$-lines (split Cartan case).
 
 $\bullet$ The normaliser of a nonsplit Cartan subgroup of $\pglep$, which means $\Prep$ is included in the normaliser of a copy of $\Fpdeux^*$ in $\GL(E_p)$.
 
 $\bullet$ An exceptional subgroup of $\pglep$,that is, a copy of $\gA_4$, $\gA_5$ or $\gS_4$ in $\pglep$ (exceptional case).
 \end{prop}

The exceptional case is immediately solved, using well-known results of \cite{Serre71} on the action of tame inertia of $E_p$.

\begin{prop}
 \label{exceptionalcase}
Let $K$ be a number field and $E$ an elliptic curve over $K$. For every prime number $p > 30 [K : \Q] + 1$, the image of $\P \rep : \GalK \ra \pglep$ is not contained in an exceptional subgroup $\Acal_4$, $\Acal_5$ or $\Scal_4$.
\end{prop}

To any $\Q$-curve $E$ for which $\Prep$ is not quasi-surjective, we can associate a point on a moduli scheme, as we will now explain. In characteristic zero, giving an isogeny on $E$ amounts to giving its kernel (up to isomorphism), hence we will not make (unless necessary) the difference between an isogeny and its kernel in the following.

$\bullet$ The scheme $X_0 (N)$ is, for any integer $N \geq 1$, the compactified coarse moduli space over $\Z$ parametrising the isomorphism classes of couples $(E,C_N)$ with $E$ an elliptic curve and $C_N$ a cyclic isogeny of degree $N$ of $E$. Its generic fiber $X_0 (N)_\Q$ is the modular curve corresponding to the congruence subgroup 
$\G_0(N)$.

$\bullet$ The scheme $X_0 ^* (N)$ is, for any integer $N \geq 1$, the quotient of  $X_0 (N)$ by its whole group of Atkin-Lehner involutions $\{w_d; d | N, (d,N/d)=1 \}$. A noncuspidal point of $X_0 ^* (N) (\Q)$ is a set of isogenous elliptic curves stable by $\GalQ$, hence the set of conjugates of one or more isogenous $\Q$-curves. This justifies the following definition borrowed to \cite{Elkies04}.

\begin{defi}
 Let $d$ be a squarefree positive integer. We call \emph{central $\Q$-curve of degree $d$} every $\Q$-curve of degree $d$ obtained from a point of $X_0 ^* (d) (\Q)$.
\end{defi}

The next proposition (reformulated from the Theorem of $\cite{Elkies04}$  with elements of its proof) allows us to see $X_0 ^* (N)$ as a sort of moduli space for $\Q$-curves of degree $N$.
\begin{prop}
\label{propElkies}
 For every $\Q$-curve $E$ without complex multiplication defined over $K$, there exists an isogeny $E \ra E'$ of degree dividing $d(E)$ towards a central $\Q$-curve $E'$ defined over $K$ and of squarefree degree $d | d(E)$.
\end{prop}

If $E$ and $E'$ are two $\Q$-curves isogenous by an isogeny of degree $m$, 
we readily see that $\Prep$ and $\Prepprime$ are isomorphic for any prime $p$ not dividing $m$. Therefore, with help of Proposition \ref{propElkies}, the problem of quasi-surjectivity of $\Prep$ for any $\Q$-curve $E$ without complex multiplication of degree $d$ prime to $p$ boils down to the same problem for central $\Q$-curves of degree $d$ and $p$ not dividing $d$. Hence, \emph{from now on and until the end of this article, we assume that every considered $\Q$-curve is a central $\Q$-curve of squarefree degree $d$ without complex multiplication}.

Let us finish with the useful moduli spaces for our problem.

$\bullet$ The scheme $\xsp$ is, for any prime $p$, the compactified coarse moduli scheme over $\Z$ parametrising the isomorphism classes of couples $(E,\{A_p,B_p\})$ with $E$ an elliptic curve and $A_p, B_p$ nonisomorphic isogenies of degree $p$ of $E$.
Its generic fiber $\xsp_\Q$ is the modular curve corresponding to the congruence subgroup 
\[
 \G_{\rm{split}} (p) : = \left\{ \g \in \slz, \g \equiv \diagstar {\rm{ or }} \antidiagstar  \! \!\mod p \right\}.
\]

$\bullet$ The scheme $\xnsp$ is, for any prime $p$, the compactified coarse moduli scheme over $\Z$ parametrising the isomorphism classes of couples $(E,\alpha_p)$
with $E$ an elliptic curve and $\alpha_p$ a copy of $\Fpdeux$ in the endomorphism ring of the group scheme $E_p$. Its generic fiber $\xnsp_\Q$ is the modular curve corresponding to one/any congruence subgroup  $\G_{\rm{nonsplit}} (p)$ which is the pullback of the normaliser of a nonsplit Cartan subgroup of $\GL_2 ( \Fp)$ in $\slz$.

If now $p$ is a prime number and $d$ a squarefree integer prime to $p$, we define (as in \cite{Ellenberg04}) the schemes
\[
 \xosdp  :=  X_0 (d) \times_{X(1)} \xsp,   \qquad
 \xonsdp  :=  X_0 (d) \times_{X(1)} \xnsp.
\]

These two schemes and $X_0(dp)$ are endowed with an involution $w_d$ which becomes the Fricke involution on $X_0(d)$ by the forgetful functors towards $X_0 (d)$. Functorially :
\[
\begin{array}{cccc}
w_d(E,C_d,C_p) & = &  (E/C_d,E_d/C_d,C_p/C_d) & \textrm{on } \quad X_0(dp),\\
w_d(E,C_d,\{A_p,B_p\}) & = & (E/C_d,E_d/C_d,\{A_p/C_d,B_p/C_d\}) & \textrm{on} \quad \xosdp,\\
w_d(E,C_d,\alpha) & = &  (E/C_d,E_d/C_d,\alpha / C_d) & \textrm{on} \quad \xonsdp,
\end{array}
\]
where, for a structure $H$ on $E_p$, $H/C_d$ means the transport of $H$ on $E/C_d$ by the isogeny $E \ra E / C_d$ (for $C_p$, it is $(C_p + C_d)/C_d$ in $E/C_d$ for example). 

Assume now that $K$ is a quadratic field with automorphism $\s$.
Let $E$ be a $\Q$-curve of degree $d$ defined over $K$ such that $\Prep$ is not surjective. The kernel of a minimal isogeny $E \ra E^\s$ is called $C_d$. 
In the Borel case, $\Prep$ stabilises some subgroup $C_p$ of order $p$ of $E_p$, therefore $P=(E,C_d,C_p)$ is a $K$-rational point on $X_0 (dp)$ such that 
\[
 P ^\s = w_d \cdot P
\]
(because we assumed $E$ is a central $\Q$-curve). Similarly, in the split Cartan case (resp. nonsplit Cartan case), we associate to $E$ a point $P$ on $\xosdp(K)$ (resp. $\xonsdp(K)$) such that $ P ^\s = w_d \cdot P$. For more details on this, see (\cite{Ellenberg04}, Proposition 2.2 and above).

\subsection{Formal immersions and Mazur's method}
\label{FormimmMazur}
Let us now recall a key proposition for Mazur's method here.

\begin{prop}
 \label{superpropimmformelle}
Let $K$ be a number field and $\lambda$ be a nonzero prime ideal of $\ok$ above $\ell$. We call $\Ocal_\lambda$ the localised ring of $\ok$ at $\lambda$ and $\F_\lambda = \ok / \lambda$.

Let $X$ be an algebraic curve defined over $\Q$ with a proper model $X_\Z$ on $\Spec \Z$, $A$ an abelian variety defined on $\Q$ with Néron model $A_\Z$ on $\Z$, and $f : X \ra A$ a morphism defined over $\Q$. It naturally extends to a morphism $f_\Z : X_\Z ^{\rm{smooth}} \ra A_\Z$ by the universal mapping property of Néron models. Now, suppose there are two points $x$ and $y$ of $X (K)$ such that : 

$\bullet$ The points $x$ and $y$ have the same reduction modulo $\lambda$, and it belongs to $X_\Z ^{\rm{smooth}}$.

$\bullet$ The morphism $f_\Z$ is a formal immersion at $x_\lambda = y_\lambda$.

$\bullet$ The point $f(y) - f(x)$ is $\Q$-rational and torsion in $A(\Q)$.

Then, if $\ell>2$, $x=y$. If $\ell=2$, either $x=y$ or $f(y) - f(x)$ is a 2-torsion point in $A(\Q)$ generating a copy of the finite group scheme $\mu_2$ in $A_\Z$.
\end{prop}

\begin{proof}
Let us suppose first that $\ell >2$.
 By hypothesis, $z= f(y) - f(x) \in A (\Q)_{tors}$. As $e=1 < \ell - 1$, according to the specialisation lemma of Raynaud (\cite{Mazur78}, Proposition 1.1), the order of $z$ is the same as the order of its reduction $z_\ell$ in $A_\Z( \F_\ell)$. Here, we have
\[
 z_\ell = z_\lambda = f_\Z (y) _\lambda - f_\Z (x)_ \lambda = f_\Z ( y_\lambda) - f_\Z ( x_\lambda ) = 0,
\]
because $x_\lambda = y_\lambda$. Therefore, $z=0$ and $f(y) = f(x)$. As $f$ is a formal immersion at $x_\lambda= y_\lambda$, this implies $x=y$.
In the case $\ell=2$, we do not have $e<\ell-1$ anymore, but thanks to Proposition 4.6 of \cite{Mazur77}, we know that $z$ is either 0 or a 2-torsion point in $A(\Q)$ generating a copy of $\mu_2$ in $A_\Z$. When $z$ is 0, the proof of the previous case works as well.
\end{proof}

\begin{rem} The main difference with Proposition 3.1 of \cite{Ellenberg04} is that the latter one did not deal with the case $\ell=2$ or the fact that $z$ is defined over $\Q$ while $x$ and $y$ are defined over a bigger field (which will be the case here). For the case $\ell=2$, we will need to rule out the case when $f(y)-f(x)$ is 2-torsion in $A ( \Q)$ to prove Proposition \ref{bonnereductionpartoutBorel}. In our study, $f(y) - f(x)$ belongs to the cuspidal subgroup $C$ of $J_0 (p)(\Q)$, and we know (\cite{Mazur77}, Proposition 11.11) its Zariski closure indeed contains a $\mu_2$ when it is of even order. This is why we actually need the analysis of the components group of the jacobian (see proof of Lemma \ref{lemcalculgP}).
\end{rem}

We recall the following classical result on  the Albanese morphism from $X_0(p)$ to $J_0(p)$ (easily obtained from the $q$-expansion principle), fundamental for Mazur's method.
\begin{prop}
 \label{immformelleAlbanese}
Let $p=11$ ou $p >13$ be a prime number. Let $\T$ be the Hecke subalgebra of $\End_\Q(J_0(p))$ generated over $\Z$ by the Hecke operators $T_n,n \in \N^*$.
Let $\phi : X_0 (p)_\Q \ra J_0 (p)_\Q$ be the Albanese morphism sending $\infty$ to 0 and $\phi_\Z : X_0 (p)_\Z ^{\rm{smooth}} \ra J_0 (p)_\Z$ its extension by Néron mapping property. For every $t \in \T$ and every prime $\ell$, $t \circ \phi_\Z$ is a formal immersion at $\infty_{\F_\ell} \in X_0 (p) ( \F_\ell)$ if and only if $t \notin \ell \T$.
\end{prop}

\subsection{Eisenstein quotient and components group of the jacobian}
\label{Eisquotcompongroup}
 The Eisenstein ideal $\Ical$ of $\T$ (defined in \cite{Mazur77}) is the ideal $\Ical := \langle 1+\ell - T_\ell, 1+w_p, l \in \Pcal_p \rangle$, with $\Pcal_p$ the set of prime numbers different from $p$.
The Eisenstein quotient $\tildeJp$ is the quotient of $J_0 (p)$ by the abelian subvariety generated by $\gamma_\Ical . J_0 (p)_\Q$, with
$\gamma_\Ical = \bigcap_{n \in \N} \Ical ^n.$
It is defined over $\Q$, and benefits the following properties (in particular, it is the first historical example of a nontrivial rank zero quotient of $J_0(p)$).
\begin{prop}
\label{propquotEisenstein}
 Let $p=11$ or $p>13$ be a prime number.
Let $n=\operatorname{num} (\frac{p-1}{12})$. Let $C$ be the cuspidal subgroup of $J_0 (p)(\Q)$, generated by $\cl ([0]-[\infty])$.

$(a)$ The rational torsion of $J_0(p)$ is exactly $C$, and it is a cyclic subgroup of order $n$ (\cite{Mazur77}, Theorem 1.2 p.142 and Proposition 11.1).

$(b)$ The canonical projection $J_0 (p) \ra \tildeJp$ is defined over $\Q$ and induces a bijection between $C=J_0 (p)(\Q)_{tors}$ and $\tildeJp (\Q)$ which is therefore a cyclic subgroup of order $n$ (\cite{Mazur77}, Corollary 1.4 p.143).

$(c)$ The Eisenstein ideal is exactly the kernel of the map $t \mt t. \cl ([0]-[\infty])$ from $\T$ to $C$, which induces an isomorphism $\T / \Ical \ra \Z / n \Z$ (\cite{Mazur77}, Proposition 11.1). In particular, when $t = 1 \mod \Ical$, $t$ acts as the identity on $C$.
\end{prop}

It is crucial for the proofs of Propositions \ref{bonnereductionpartoutBorel} and \ref{BornebonnereductioncasCartandeploye} for prime ideals above 2 to cancel $\gamma_\Ical$ by ``good elements of $\T$''. The following lemma will allow us to do so.

\begin{lem}
\label{Nakayamadonneannulateurgammaical}
For every prime number $\ell$, there exists a Hecke operator $t \in \T \backslash \ell \T$ such that $t = 1 \mod \Ical$ and $t \cdot \gamma_\Ical=0$. 
Moreover, for such a $t \in \T$,  $t \cdot (1 + w_p)=0$.
\end{lem}

\begin{proof}
As $\T$ is a noetherian ring, by Artin-Rees lemma, we have $\Ical \cdot \gamma_\Ical = \gamma_\Ical$. Hence, by Nakayama's lemma, $\gamma_\Ical$ is cancelled by some element $t \in \T$ congruent to 1 mod $\Ical$. For a fixed prime $\ell$, we can even choose such a $t$ not in $\ell \T$ : if $\ell$ divides $n$, this is automatic because of Proposition \ref{propquotEisenstein} $(c)$, otherwise we can choose an integer $\ell'$ prime to $\ell$ but congruent to $\ell \mod n$. Then, for $k$ such that $t \in \ell ^k \T \backslash \ell^{k+1} \T$, the operator $t' = (\ell ' / \ell) ^k t  \in \T \backslash \ell \T$ but is still congruent to $1 \mod \Ical$ while cancelling $\gamma_\Ical$.
Finally, such a $t \in \T$ automatically cancels $(1 + w_p)$ under the previous conditions, because the Eisenstein quotient is a quotient of the minus part of the jacobian $J_0 (p)$ (\cite{Mazur77}, Chapter 2, Proposition 17.10).
\end{proof}

We need to describe the components group of the special fiber of $J_0 (p)_R$, which is made possible by Theorem 9.6.1 of \cite{BoschRaynaud}. 
In short, if $R$ is a discrete valuation ring of mixed characteristic with fraction field $K$ and perfect residual field $k$, and $\Phi$ is the group of irreducible components of $J_0(p)_R \otimes \overline{k}$, this theorem gives a description of $\Phi$ by generators $(\Ccal)$, and relations given by the  irreducible components of $\widetilde{X_0(p)_R}$ (the minimal regular model of $X_0(p)$ over $R$) and their intersection numbers, and this description is compatible with the reduction morphism from $J_0(p)(K)$ to $J_0(p)_R \otimes \overline{k}$.

A first application of this to $J_0 (p)_\Z$ gives the following results (\cite{Mazur77}, Theorem 10 and Appendix).
\begin{prop}
\label{Mazure1}
 Let $p=11$ or $p>13$ be a prime number and $n = \operatorname{num}  \left( \frac{p-1}{12} \right)$. Let $\phi : X_0 (p)_\Q \ra J_0 (p)_\Q$ be the Albanese morphism sending $\infty$ to 0.

$(a)$  Reduction modulo $p$ of the cuspidal subgroup $C=\langle \cl([0]-[\infty]) \rangle$ induces an isomorphism from $C$ to the group of components $\Phi$ of $J_0 (p)_\Z \otimes \Fpb$.

$(b)$ For every point $Q \in J_0 (p)( \Q)$, we define $\rho (Q) \in \Z/ n \Z$ the image of $Q$ by 
\[
J_0 (p) ( \Q) \ra J_0(p)_\Z (\Fpb) \ra \Phi \cong C \cong \Z / n \Z.                                                                                                          
\]                                    
Then, for every point $P \in Y_0 (p) (\Q)$ :                                                                      

$\bullet$ If $E$ has potentially ordinary or multiplicative reduction modulo $p$, $\rho(\phi(P))=0$ if $C_p$ defines a separable isogeny modulo $p$, and $\rho(\phi (P)) = 1$ otherwise.

$\bullet$ If $E$ has potentially supersingular reduction modulo $p$, either $p = -1 \mod 4$, $j(E)=0 \mod p$ and
 $\rho (\phi(P)) = 1/2$, or $p=-1 \mod 3$, $j(E)=1728 \mod p$  and $\rho (\phi(P)) = 1/3$ or $2/3$.
\end{prop}

The idea underlying $(b)$ of Proposition is that $\rho \circ \phi ( X_0(p) ( \Q))$ is small, and in particular it does not not the unique nontrivial 2-torsion point of $C$ for $p=11$ or $p>13$. 
This is the same idea that we will use in a more general analysis in Lemma \ref{lemcalculgP}

Let $R$ be a complete discrete valuation ring of characteristic 0, with fraction field $K$ and residual field $k$ that we assume perfect of characteristic $p$. Let $\pi$ be a uniformizer of $R$, $v$ be the normalised valuation on $K$ and $e=v(p)$ the absolute ramification of $R$.

The following result of Edixhoven (section 3 of the Appendix of \cite{BertoliniDarmonEdixhoven}) generalises Theorem 1.1 of the Appendix of \cite{Mazur77}.
\begin{prop}
 $(a)$ The scheme $X_0 (p)_R$ is smooth over $R$ outside its supersingular points in the special fiber.

$(b)$ The geometric special fiber $X_0 (p)_{\overline{k}}$ is made up with two copies of $\P^1 (\overline{k})$ (which are also its irreducible components) crossing transversally at supersingular elliptic curves : the first one, called $Z$, parametrises elliptic curves endowed with their Frobenius isogeny, and the second one, called $Z'$, parametrises elliptic curves endowed with their Vershiebung isogeny.

$(c)$ Let $s$ be a supersingular point of $X_0 (p)_R \otimes_R \overline{k}$ corresponding to a couple $(E,C_p)$ with $E$ an elliptic curve over $\overline{k}$ and $C_p$ a $p$-isogeny of $E$. We call width of $s$ the integer $k_s =\! |\! \Aut (E,C_p)|/2$.
The scheme $X_0(p)_R$
is nonregular at $s$ if and only if $e k_s >1$. More precisely, the local completed ring of $X_0(p)_R$ at $s$ is isomorphic to $R [[X,Y]] / (X Y - \pi ^{ek_s})$.

If $p \neq 2,3$, then $k_s>1$ implies that $j(E) = 0$, $k_s=3$ and $p = - 1\mod 3$, or $j(E)=1728$, $k_s=2$ and $p = -1 \mod 4$.

$(d)$ The geometric fiber $(\widetilde{X_0 (p)_R})_{\overline{k}}$ of the minimal regular model over $R$ is therefore obtained by blowing up in $X_0 (p)_{\overline{k}}$ every nonregular point $s$ to a chain of $ek_s-1$ projective lines. These projective lines, as Cartier divisors, have auto-intersection $-2$.
\end{prop}

For the following proposition on components group, we need some notations. First, we can suppose $e>1$ because the case $e=1$ is dealt with Proposition \ref{Mazure1}. Therefore, every supersingular point $s$ in the special fiber \emph{is} nonregular. 
To provide some intuition on the proof, we define (as in the Appendix of \cite{BertoliniDarmonEdixhoven}) the dual graph $\widetilde{G}$ associated to $\widetilde{X_0(p)_R}_{\overline{k}}$ : its vertices are the irreducible components of $(\widetilde{X_0 (p)_R})_{\overline{k}}$ and an edge links two vertices if and only if the two components intersect. Thanks to Theorem 9.6.1 of \cite{BoschRaynaud}, the problem therefore becomes a problem on $\widetilde{G}$ : compute the abelian group $\Phi$ given by generators (its vertices) and relations (the image of the laplacian operator on the graph).

We define $\Scal$ (resp. $\Scal'$) the set of cardinality $S$ (resp. $S'$) of supersingular points of $X_0 (p)_{\overline{k}}$ (resp. supersingular points with $j$-invariant different from 0 and 1728).
We also define $I=1$ if the elliptic curve with $j$-invariant $1728$ is supersingular in $k$, 0 otherwise and $R=1$ if the elliptic curve with $j$-invariant $0$ is supersingular, 0 otherwise, so that
\begin{equation}
\label{massformula}
 S = S' + I + R \quad {\textrm{ and }} \quad  S' + \frac{I}{2} + \frac{R}{3} = \frac{p-1}{12}.
\end{equation}
from (\cite{SilvermanAEC}, Theorem $V$.4.1 $(c)$).

For every $s \in \Scal$, we call $\Ccal_s$ the path of length $e k_s$ between the points $Z$ and $Z'$ associated to $s \in \Scal$ in $\widetilde{G}$, due to the blowup of $s$ in  $\widetilde{X_0 (p)_R}$. In case $s$ is of $j$-invariant 1728 (resp. 0), we also call it $\Ecal$ (resp. $\Gcal$). We order the points of $\widetilde{G}$ (that is, the components of $\widetilde{X_0(p)_R} \times \overline{k}$) on every one of these paths in the following way : we call $C_{s,0} = Z'$, $C_{s,1}=C_s$ the unique point of $\Ccal_s$ linked to $Z'$, $C_{s,2}$ the point of $\Ccal_s$ linked to $C_s$ not yet named, and so on until  $C_{s,ek_s} = Z$. If $s$ is of $j$-invariant $1728$ (resp. 0), we call $E=\Ccal_{s,1}$ (resp. $G = \Ccal_{s,1}$) to remain consistent with notations of the Appendix of \cite{Mazur77} (where $F= \Ccal_{s,2}$ with $j(s) = 0$).

We use for every irreducible component $C$ the notation
\[
 \overline{C} = [C] - [Z']
\]
as $Z'$ is the component where reduces the cusp $\infty$, our choice of base point for the Albanese morphism. The following lemma simplifies the presentation of $\Phi$.

\begin{lem}
 Let $\Ccal_s$ be the path of length $e k_s$ between $Z$ and $Z'$ associated to $s$ in $\widetilde{G}$. In the group $\Phi$, for every $i \in [|0,e k_s|]$, 
\[
 \overline{C_{s,i}} = i \overline{C_{s,1}} = i \overline{C_s}
\]
In particular, 
\[
 \overline{Z} = e k_s \overline{C_s}.
\]

\end{lem}

\begin{proof}
 It is true by definition for $i=0$ and 1. For every $i \in [|1,m-1|]$, the relation given by the laplacian operator on $\widetilde{G}$ at $C_{s,i}$ is
\[
  - 2 \overline{C_{s,i}} + \overline{C_{s,i-1}} + \overline{C_{s,i+1}} = 0.
\]
The lemma follows by double induction on $i$.
\end{proof}

The group $\Phi$ is therefore the abelian group generated by $\overline{Z}$ and the $\overline{C_s}$, $s \in \Scal$, and the relations
\[
\begin{array}{rclr}
 - S \overline{Z} + (2e-1) I \overline{E} +  (3 e-1) R \overline{G} + \sum_{s \in \Scal'} (e-1) \overline{C_s} & = & 0& \qquad (Z) \\
 I \overline{E} + R \overline{G} + \sum_{s \in \Scal'} \overline{C_s} & = & 0 & \qquad (Z') \\
 \forall s \in \Scal', \qquad  \overline{Z} &= & e \overline{C_s} ,& \qquad (C_s)
\end{array}
\]
with additional relations 
 \[
\begin{array}{rclr}
 \overline{Z} & = & 2e \overline{E}& \qquad (E)\\
 \overline{Z} & = & 3 e \overline{G} & \qquad (G)
\end{array}
\]
whenever $E$ or $G$ exist (the name of these relations corresponding to the point where the laplacian operator is applied). Adding to $(Z)$ the relation $-(e-1) (Z')$, we get a new relation
\[
 - S \overline{Z} + e I \overline{E} + 2 e R \overline{G} = 0 \qquad (Z).
\]
that we use in replacement of the previous relation $(Z)$. Let us now give the description of the components group $\Phi$.
\begin{prop}
\label{groupecomposantescasgeneral}
 Let $p=11$ or $p>13$ be a prime number and $n=\operatorname{num}  \left( \frac{p-1}{12} \right)$. 
With previous notations : 

$(a)$ The group of irreducible components $\Phi$ of $J_0 (p)_R \times {\overline{k}}$ is non-canonically of the form
\[
 \Phi \cong (\Z / n e \Z) \times (\Z / e \Z) ^{S-2}.
\]

$(b)$ The cuspidal subgroup $C$ of $J_0 (p) (\Q)$ reduces injectively in $\Phi$, with $\overline{Z}$ being the reduction of $\cl ([0]-[\infty])$. Hence, $\langle \overline{Z} \rangle$ is of order $n$ and we identify it with $\Z / n \Z$ in the following.

$(c)$ There is an exact sequence of $\Z/ e \Z$-modules  
\[
\xymatrix{0 \ar[r] & \Z/ e \Z \ar[r] ^-\Delta & \left( \Z/ e \Z \right) ^S \ar[r]^-\alpha & \Phi / \langle \overline{Z} \rangle \ar[r] & 0}
\]
with $\Delta :  \lambda \mt \lambda \cdot \sum_{s \in S} [C_s]$ and $\alpha :  \sum_{s \in S} \lambda_s [C_s] \mt \sum_{s \in S} \lambda_s \overline{C_s}$.

In particular, $e \cdot \Phi =\langle \overline{Z} \rangle$, and we get  
\[
\begin{array}{rcl}
\forall s \in \Scal',\forall i \in [|1,e-1|], e. \overline{C_{s,i}} & = & i \\
\forall i \in [| 1,2e-1 |], e . \overline{E_i} & = & i/2 \qquad {\rm{ if }} \, p= -1 \mod 4\\
\forall i \in [|1,3 e -1 |], e . \overline{G_i} & = & i/3 \qquad {\rm{ if }} \, p= -1 \mod 3\\
\sum_{s \in \Scal} \overline{C_s} & = & 0 \qquad {\rm{ in }} \Phi
\end{array}
\]

\end{prop}

\begin{proof}
We only compute $\Phi$ in the case $p=11 \! \mod 12$ so that $I=R=1$, the other cases being simpler but similar.
We replace $\overline{Z}$ by $2 e \overline{E}$ thanks to relation $(E)$, and for every $s \in \Scal'$ we make the variable changes
\[
 \overline{C_s '} : = \overline{C_s} - 2 \overline{E} \quad {\textrm{and}} \quad \overline{G'} : = \overline{G} + (2 S - 3) \overline{E}. 
\]
The relations become
\[
 \begin{array}{rclr}
  e (6S-7) \overline{E} + 2 e \overline{G'} & = & 0 &\qquad (Z)' \\
  \overline{G'}  + \sum_{s \in \Scal'} \overline{C_s'} & = & 0 &\qquad (Z') ' \\
 \forall s \in \Scal', e . \overline{C_s '} & = & 0 & \qquad (C_s) '
 \end{array}
\]
As $p= 11 \mod 12$, $6S-7= 6 S' + 5= (p-1)/2=n$ by the mass formula of \eqref{massformula}, and $(Z)'$ is equivalent to $e n \overline{E}=0$ with help of the other relations.
Hence, $\overline{G'}$ is generated by the other generators, and the relations $(C_s) ', s \in \Scal'$ and $(Z)'$ are diagonal, which gives us the announced isomorphism of $(a)$. Via this isomorphism, we have
\begin{eqnarray*}
 \overline{Z} & = & (2e,0, \cdots, 0) \\
 \overline{E_i} & = & (i, 0 \cdots , 0) \\
 \overline{G_i} & = & (-(2 S' + 1) i, -i, \cdots, -i) \\
\overline{C_{s,i}} & = & (2i, 0 ,\cdots,0, i,0 , \cdots, 0) \quad (s \in \Scal') \\
\end{eqnarray*}
and this readily gives us $(b)$ and $(c)$.
\end{proof}

\section{Application of Mazur's method}
\subsection{Borel case}
\label{Borel case}
Let $p$ be a fixed prime and $d$ be a squarefree positive integer prime to $p$. The scheme $X_0 (dp)_\Z$ is smooth outside its supersingular points in characteristic dividing $dp$, in particular every cusp reduces in the smooth part modulo every prime number $\ell$.
If $r$ is the number of prime factors of $d$, $X_0 (dp)_\Q$ has $2 ^{r+1}$ cusps on which the Atkin-Lehner group acts transitively.
As $d$ and $p$ are coprime, the cusps of $X_0 (dp)_\Q$ are in bijection via forgetful functors $X_0 (dp)_\Q \ra X_0 (d)_\Q$ and $X_0 (dp)_\Q \ra X_0 (p) _\Q$ with couples of cusps of $X_0 (d)_\Q$ and $X_0 (p)_\Q$. Via this correspondance, the Atkin-Lehner involutions $w_{d'}, d' |d$ on $X_0 (dp)_\Q$ leave unchanged the component of the cusp in $X_0 (p)_\Q$ and $w_p$ leaves unchanged the component in $X_0 (d)_\Q$. 
We call $\infty^{dp}$ (resp. $\infty ^d$, $\infty^p$) the usual infinity cusp of $X_0 (dp)_\Q$ (resp. $X_0(d)_\Q$, $X_0 (p)_\Q$).

\begin{defi}
 We denote by $\pi_{dp,p} : X_0 (dp)_\Q \ra X_0 (p)_\Q$ the ``forgetting $d$-structure'' morphism and by $\phi : X_0 (p)_\Q \ra J_0 (p)_\Q$ the Albanese morphism sending $\infty^p$ to 0. The morphism ${g : X_0 (dp)_\Q \ra J_0 (p)_\Q}$ is defined by $
 g:= \phi \circ \pi_{dp,p} + \phi \circ \pi_{dp,p} \circ w_d$.
Functorially, we have
\[
 g(E,C_d,C_p) = \cl ([E,C_p] + [E/C_d,C_p/C_d] - 2 [\infty^p]). 
\]
For every $t \in \T$, we note $g_t = t \circ g$.
\end{defi}

\begin{prop}
\label{formalimmersionBorel}
 Let $p=11$ or $p>13$ be a prime number, and $\ell$ be a prime number (possibly equal to $p$).
For every $t \in \T$, the morphism $(g_t)_\Z : X_0 (dp) ^{\rm{smooth}} _\Z \ra J_0 (p)_\Z$ extending $g_t$ by Néron mapping property is a formal immersion at $\infty_{\F_\ell} ^{dp}$ if and only if $t \notin \ell \T$.
\end{prop}

\begin{proof}
 Fix $t \in \T$ and $g'_t = t \circ \phi \circ \pi$ so that $g_t = g'_t + g'_t \circ w_d$. As $w_d$ permutes cusps of $X_0 (dp)_\Q$ above $\infty^p \in X_0 (p)$, $g'_t ( \infty ^{dp} ) = g'_t \circ w_d ( \infty^{dp}) = 0$. The residual fields of the points $\infty^{dp}_\Fl$ and $0_\Fl$ are both $\Fl$ as $\infty^{dp}$ and $0 \in J_0 (p)$ cusps are $\Q$-rational, so we only have to check the induced maps on cotangent spaces. Notice $w_d ( \infty ^{dp})$ is a cusp of $X_0 (dp)_\Q$ not above $\infty ^d$, therefore $({\pi_{dp,p}})_ \C : X_0 (dp)_\C \ra X_0 (p)_\C$ is ramified at this cusp (because $X_0(d)_\C \ra X(1)_\C$ is ramified at any cusp but $\infty ^d$). Hence the cotangent map of $g'_t \circ w_d$ is zero at section $\infty ^{dp}_\Z$ and the cotangent map of $g_t$ at $\infty ^{dp}_\Z$ is the cotangent map of $g'_t$. Finally the cotangent map of $({\pi_{dp,p}})_\Z$ at section $\infty_\Z ^{dp}$ is an isomorphism so the cotangent map of $g'_t$ at $\infty^{dp}_\Fl$ is surjective if and only if $t \notin \ell \
T$ by Proposition \ref{immformelleAlbanese}, which concludes the proof by the usual formal immersion criterion.
\end{proof}

The following lemma is essential for the case $\ell=2$.
\begin{lem}
\label{lemcalculgP}
 Let $p=11$ or $p>13$ be a prime number, $n = \num ((p-1)/12)$ and $K$ be a quadratic field with ramification degree $e$ over $p$.
Let $E$ be a central $\Q$-curve of degree $d$ defined over $K$ such that $\Prep$ is reducible, and $P$ the corresponding point of $X_0 (dp) (K)$.
Then $g(P)$ belongs to $J_0 (p) (\Q)$, and with the retraction $\rho : J_0(p)(\Q) \ra \Z/n\Z $defined at Proposition \ref{Mazure1}, the possible values of $\rho ( g (P))$ are the following : 

\begin{tabular}{|c|c|c|c|c|}
\hline
& $p=1 \mod 12$ & $p=5 \mod 12$ & $p=7 \mod 12$ & $p=11 \mod 12$ \\
\hline
$e=1$  & 0,2 & 0,1,2,$\frac{2}{3},\frac{4}{3}$  & 0,1,2  & 0,1,2,$\frac{2}{3},\frac{4}{3}$\\
\hline
$e=2$ &0,1,2 & 0,1,2,$\frac{1}{3},\frac{2}{3},\frac{4}{3},\frac{5}{3}$ & 0,1,2,$\frac{1}{2},\frac{3}{2}$ & 0,1,2,$\frac{1}{3},\frac{1}{2},\frac{2}{3},\frac{4}{3},\frac{3}{2},\frac{5}{3}$\\
\hline
\end{tabular}

In particular, this image cannot be the unique 2-torsion point of the cuspidal subgroup $C$ unless $p = 17$ or $41$.
\end{lem}

\begin{proof}
 First, $g(P) \in J_0 (p) (K)$ because $g$ is $\Q$-rational, and for $\s$ the nontrivial automorphism of $K$, 
\[
 g(P) ^\s = \phi \circ \pi_{dp,p} ( P ^\s ) + \phi \circ \pi_{dp,p} \circ w_d  (P ^\s) = \phi \circ \pi_{dp,p} (w_d \cdot P  ) + \phi \circ \pi_{dp,p} (P) = g(P)
\]
by construction of $P$ (see subsection \ref{rappelsrepproj}), hence $g(P)$ indeed belongs to $J_0 (p) (\Q)$. Fix a prime ideal $\gP$ of $K$ above $p$ and define $\Phi$ the group of components of the special fiber of $J_0(p)_{{\ok}_\gP}$. We read $\rho (g(P))$ via reduction modulo $\gP$, using Propositions \ref{Mazure1} and \ref{groupecomposantescasgeneral} (and their notations). Notice that $\pi_{dp,p} (P)$ and $\pi_{dp,p} (w_d \cdot P)$ represent  elliptic curves which are isogenous of degree prime to $p$ on $X_0 (p)$. Consequently, their stable reduction type modulo $\gP$ (ordinary with separable isogeny, ordinary with unseparable isogeny, supersingular) is the same. 
This remark done, we use the Proposition \ref{groupecomposantescasgeneral} to compute the different possibilities : 

$\bullet$ If $p$ is unramified in $\ok$, we are in the étale case of the Proposition, and the possible values of the reduction of $g(P)$ in $\Phi$ are then $2 \overline{Z'} =0$, $2 \overline{Z} =2$, $2 \overline{E} = 1$, $2 \overline{G} = 2/3$, $4 \overline{G} = 4/3$ and $3 \overline{G} = 1$ (when $\overline{E}$ or $\overline{G}$ exist). This proves the first line.

$\bullet$ If $p$ is ramified in $\ok$, we are in case $e=2$ of the Proposition, and the possible values of the reduction of $g(P)$ in $\Phi$ are then $2 \overline{Z'} =0$, $2 \overline{Z} =2$, $2 \overline{E_1} = 1/2$, $2 \overline{E_2}=1$, $2 \overline{E_3}=3/2$, $\overline{E_1} + \overline{E_3} = 1$, $2 \overline{G_1} = 1/3$, $2 \overline{G_2}=2/3, \dots, 2 \overline{G_5} = 5/3$ (and other possibilities that do not give new values). Finally, for every $s \in \Scal'$, $2 \overline{C_{s,1}} = 1$.

For the application of the result, $\Z / n \Z$ admits a nontrivial 2-torsion point if and only if $2$ divides $n$, which is possible only when $p = 1 \mod 8$, and then this torsion point is $n/2$. Suppose $\rho (g(P))$ is equal to $n/2$. From the table, if $p=1 \mod 12$ , it means $(p-1)/12$ divides 1,2 or 4, whence $p=13$. If $p = 5 \mod 12$, it means $(p-1)/4$ divides $1,2,4,8$ or 10, whence $p=17$ or $p=41$.
\end{proof}

\begin{prop}
\label{bonnereductionpartoutBorel}
 Let $K$ be a quadratic field and $E$ a $\Q$-curve of squarefree degree $d$ defined over $K$. If for $p \geq 11$, $p \neq 13,17,41$ prime not dividing $d$, the representation $\Prep$ is reducible, $E$ has potentially good reduction at every prime ideal of $\ok$. 
\end{prop}

\begin{proof}
 Let $P \in X_0 (dp) (K)$ be the point associated to $E$. From Lemma \ref{lemcalculgP}, we know that $g(P) \in J_0 (p)(\Q)$.
Let $\lambda$ be a prime ideal of $\ok$ above $\ell$, and suppose $E$ has potentially multiplicative reduction at $\lambda$. As the Atkin-Lehner involutions group acts transitively and $\Q$-rationally on the cusps of $X_0 (dp)$, we can and will assume that $P_\lambda = \infty^{dp}_\lambda = \infty^{dp}_\ell \in X_0 (dp)_\Z ( \Fl)$. In particular, $P_\lambda$ is in the smooth part of $X_0 (dp)_\Z$.
We choose $t \in \T \backslash \ell \T$ cancelling the ideal $\gamma_\Ical$ and congruent to 1 modulo $\Ical$ (see Lemma \ref{Nakayamadonneannulateurgammaical} and above for notations). The $\Q$-rational morphism $t : J_0 (p) _\Q \ra J_0 (p)_\Q$ cancels on $\gamma_\Ical \cdot J_0 (p)$, therefore it factors through $\widetilde{J} (p)$ and sends the $\Q$-rational point $g(P)$ on a $\Q$-rational torsion point of $J_0 (p)$ because $\widetilde{J}(p)$ is of rank zero. It allows us to apply Proposition \ref{superpropimmformelle} to $x=P$, $y= \infty^{dp}$, $X= X_0 (dp)_\Q$, $A=J_0 (p)_\Q$ and $g_t$ (the latter being a formal immersion by Proposition \ref{formalimmersionBorel}). In the case $\ell>2$, we obtain $P = \infty ^{dp}$, which is a contradiction, therefore $E$ has potentially good reduction modulo $\lambda$.
In the case $\ell=2$, as $t= 1 \! \mod \Ical$, $\rho ( g_t (P)) = \rho (g (P))$ (Proposition \ref{propquotEisenstein} $(c)$), therefore $g_t (P)$ is not the non-trivial 2-torsion point of $J_0 (p) (\Q)$ by Lemma \ref{lemcalculgP}. So $P= \infty^{dp}$ in this case too, which is a contradiction, hence $E$ has potentially good reduction modulo $\lambda$ for every prime ideal $\lambda$ of $\ok$.
\end{proof}

\subsection{Split Cartan case}
\label{Split Cartan case}
Let $p$ be a prime number and $d$ be a squarefree positive integer prime to $p$.
As in \cite{Ellenberg04}, we define $\xspcar _\Z$ the compactified coarse moduli scheme over $\Z$ parametrising the triples $(E,A_p,B_p)$ with $A_p,B_p$ nonisomorphic $p$-isogenies of $E$ and $\xospcardp_\Z : = X_0 (d) \times_{X(1)} \xspcar _\Z$
  . These schemes are proper over $\Spec \Z$, and there is a natural involution
\[
 w: (E,C_d,A_p,B_p) \ra (E,C_d,B_p,A_p)
\]
so that the quotient $\xospcardp_\Z /  \langle w \rangle $ is the scheme $\xosdp _\Z$ defined in subsection \ref{rappelsrepproj}. As for Borel case, the ``forgetting $d$-structure'' and ``forgetting $p$-structure'' morphisms induce a bijection between cusps of $\xospcardp_\Q$ (resp. $\xosdp_\Q$) and couples of cusps of $X_0 (d)_\Q$ and cusps of $\xspcar_\Q$ (resp. of $\xsp_\Q$) because $d$ is prime to $p$. Henceforth, we will note $(c,c')$ the unique cusp of $\xospcardp_\Q$ (resp. $\xosdp_\Q$) above the cusps $c \in X_0 (d)_\Q$ and $c' \in  \xspcar_\Q$ (resp. $c' \in \xsp$). The scheme $\xospcardp_\Z$ is smooth in characteristics prime to $dp$. In characteristic $\ell$ dividing $d$, its singular points correspond to supersingular elliptic curves over $\overline{\Fl}$. In characteristic $p$, $\xospcardp_\Fpb$ (resp. $\xspcar_\Fpb$) is made up with three irreducible components which parametrise the quadruples $(E,C_d,A_p,B_p)$ (resp. $(E,A_p,B_p)$) such that respectively :

$\bullet$ The isogeny $A_p$ is unseparable ($Z$ component).

$\bullet$ The isogeny $B_p$ is unseparable ($Z'$ component).

$\bullet$ Neither $A_p$ nor $B_p$ are unseparable ($W$ component).

Cusps above $\infty^p \in \xspcar$ reduce in component $Z$, cusps above $0^p \in \xspcar$ reduce in component $Z'$, and all the others reduce to component $W$, which is of multiplicity $p-1$. Hence, the nonrational cusps do not reduce to smooth points.
Considering the quotient morphism $\xospcardp_\Z \ra \xosdp_\Z$ (resp. $\xspcar_\Z \ra \xsp_\Z$), the components $Z$ and $Z'$ identify in the fiber at $p$ to form a component called $Z_0$, the other irreducible component being the image of the component $W$ (still noted $W$).

We need preparatory results inspired from \cite{Momose84}. We also owe this article the original idea for the following formal immersion, that \cite{Ellenberg04} adapted for $\Q$-curves.

\begin{prop}
\label{reduccompcartdeploye}
 Let $E$ be an elliptic curve defined over a number field $K$ and $p > 2 [K : \Q] + 1$ be a prime number such that the image of $\rep$ is in the normaliser of a split Cartan subgroup of $\glep$. Let $P = (E,\{ A_p,B_p \})$ be the corresponding $K$-rational point of $\xsp$, and $K'$ an extension of degree two of $K$ over which $A_p$ and $B_p$ are defined. Then, for any prime ideal $\gP$ of $\ok$ above $p$ : 

$(a)$ The elliptic curve $E$ does not have potentially supersingular reduction at $\gP$.

$(b)$ The reduction modulo $\gP$ of $P$ in $\xsp_\Fp$ does not belong to the $W$ component.

$(c)$ For every prime ideal $\gP '$ of $\Ocal_{K'}$ above $\gP$, the reductions modulo $\gP'$ of $(E,A_p)$ and $(E/B_p, E_p/B_p)$ belong to the same irreducible component of $X_0 (p)_{\Ocal_{\gP'}} \times \Fpb$.
\end{prop}

\begin{proof}
Let $p \geq 5$ be a prime number and $\gP'$ be a prime ideal of $\Ocal' = \Ocal_{K'}$ above $p$.
The $\Gal ( \Qb / K')$-modules $E_p$, $A_p$ and $B_p$ define group schemes over $K'$ noted $(E_p)_{K'}$, $(A_p)_{K'}$ and $(B_p)_{K'}$ such that
\begin{equation}
\label{isomepapbp}
  (E_p)_{K'} \cong (A_p)_{K'} \oplus (B_p)_{K'}.
\end{equation}
Passing to Zariski closure in Néron model of $E$ over $\Ocal'$, we obtain group schemes $(E_p)_{\Ocal'}$, $(A_p)_{\Ocal'}$ and $(B_p)_{\Ocal'}$ extending the group schemes over $K'$. As the absolute ramification index $e'$ of $p$ in $\Ocal_{K'}$ is smaller that $2 [K : \Q] <p-1$, we know by Raynaud's specialisation lemma (Theorem 3.3.3 and Corollary 3.3.4 of \cite{Raynaud74}) that
 \[
 {(E_p)}_{\Ocal '} \cong {(A_p)}_{\Ocal'} \oplus {(B_p)}_{\Ocal'}.
\]
Furthermore, as $e' < p+1$ again, the finite group schemes $(A_p)_{\Ocal'}$ and $(B_p)_{\Ocal'}$ are constant or isomorphic to $\mu_p$. This already proves $(a)$ because if $E$ is potentially supersingular, $(E_p)_{\Ocal '}$ contains a group scheme $\alpha_p$. Furthermore, ${(E_p)}_{\Ocal '}$ is not étale since it has at most $p$ $\Fpb$-rational points while being of rank $p ^2$. Therefore, $(A_p)_{\Ocal'}$ or $(B_p)_{\Ocal'}$ is not étale, hence isomorphic to $\mu_p$. This proves $(b)$. Finally, the $\Gal ( \Qb / K')$-modules isomorphism $A_p \ra E_p / B_p$ given by \eqref{isomepapbp} extends (as $e'<p-1$ again) to an isomorphism ${(A_p)}_{\Ocal'} \cong {(E_p/B_p)}_{\Ocal'}$. Hence, they are simultaneously constant or isomorphic to $\mu_p$. As $E$ is not potentially supersingular, the component to which belongs the reduction of $(E,A)$ in $X_0 (p)_{\Ocal'} \times \Fpb$ is $Z$ or $Z'$, and entirely determined by the nature (étale or not) of $A$, hence the two points $(E,A_p)$ and $(E/B_p, E_p/B_p)$ reduce to 
the same component.

\end{proof}

Recall $\phi : X_0(p)_\Q \ra J_0 (p)_\Q$ is the Albanese morphism sending $\infty^p$ to 0.
\begin{defi}
Let $p$ be a prime number.
We note $\pi : \xospcardp_\Q \ra X_0 (p)_\Q$ the forgetful morphism that sends $(E,C_d,A_p,B_p)$ to $(E,A_p)$.  The map $h : \xospcardp_\Q \ra J_0 (p)_\Q$ is defined by $h: = \phi \circ \pi - \phi \circ w_p \circ (\pi \circ w) + \phi \circ \pi \circ w_d - \phi \circ w_p \circ (\pi \circ w \circ w_d)$. Functorially, we have
\[
h(E,C_d,A_p,B_p) = \cl \left( [E,A_p] - [E/B_p,E_p/B_p] + [E/C_d,A_p/C_d] - [E/(B_p+C_d),E_p/ (B_p + C_d)] \right).
\]
Furthermore, $h \circ w = - w_p \circ h$ where $w_p$ is the endomorphism of $J_0 (p)_\Q$ corresponding to Atkin-Lehner involution $w_p$ on $X_0 (p)_\Q$. For every $t \in \T$, we define $h_t = t \circ h$. Hence, if $t( 1+ w_p) = 0$, $h_t \circ w = h_t$ so that $h_t$ factors through the projection $\xospcardp_\Q \ra \xosdp_\Q$ in a $\Q$-rational morphism $h_t ^+ : \xosdp_\Q \ra J_0 (p)_\Q$.
\end{defi}

\begin{rem}
The sum $h'$ of the first two terms of $h$ give the good candidate for elliptic curves over $\Q$ (see \cite{BiluParent11}). As in Borel case, the generalisation to $\Q$-curves relies in the consideration of $h' + h' \circ w_d$.
\end{rem}

\begin{prop}
\label{propimmformellecasdeploye}
Let $\ell$ be a prime number and $t \in \T$.

  The morphism
$(h_t)_\Z : \xospcardp  ^{\rm{smooth}} \ra J_0 (p)_\Z$ extending $h_t$ by Néron mapping property is a formal immersion at the cusp $(\infty,\infty)_{\F_\ell}$ of $\xospcardp_{\F_\ell}$ if and only if $t \notin \ell \T$. If $t ( 1 + w_p) = 0$, the same condition holds for $(h_t)_\Z ^+ : \xosdp ^{\rm{smooth}} \ra J_0 (p)_\Z$ to be a formal immersion at the cusp $(\infty,\infty)_{\F_\ell}$ of $\xosdp_\Fl$.
\end{prop}

\begin{proof}
 We define $\psi : \xospcardp_\Q \ra X_0 (dp)_\Q$ as the ``forgetting $B_p$''-morphism that sends $(E,C_d,A_p,B_p)$ to $(E,C_d,A_p)$. With the definitions of the Borel case, $\pi = \pi_{dp,p} \circ \psi$ and $
 h = g \circ \psi + g \circ w_p \circ \psi \circ w.
$
As $\psi (( \infty,\infty)) = \infty^{dp} = w_p \circ \psi \circ w (\infty,\infty)$, the cotangent map of $h_\Z$ at section $(\infty,\infty)_\Z$ is the sum of the cotangent maps of $(g \circ \psi)_\Z$ and $(g \circ w_p \circ \psi \circ w)_\Z$. Notice that $\psi$ is ramified of degree $p$ at $(\infty,\infty)_\Z$ (as it can be checked out on corresponding Riemann surfaces), so that $g \circ \psi$ is. Hence, the cotangent map of $h_\Z$ is the cotangent map of $(g \circ w_p \circ \psi \circ w)_\Z$. Furthermore, we readily see that $g \circ w_p = w_p \circ g  + 2 \cl ( [0] - [\infty^p])$, so $(h_t)_\Z$ is a formal immersion at $(\infty,\infty)_{\F_\ell}$ if and only if $(g_t \circ \psi \circ w)_\Z$ is, because $w_p$ is an automorphism of $J_0 (p)_\Z$. But the cotangent map of $(\psi \circ w)_\Z : \xospcardp_\Z \ra X_0 (p)_\Z$ at the section $(\infty,\infty)_\Z$ is an isomorphism (see \cite{Momose84}, Proof of Proposition 2.5), whence the result for $(h_t)_\Z$ by Proposition \ref{formalimmersionBorel}. The result 
follows for $(h_t)_\Z ^+$ as $(\infty,\infty)$ is not a fixed point of $w$.
\end{proof}

We can now prove our result in split Cartan case.
\begin{prop}
\label{BornebonnereductioncasCartandeploye}
Let $K$ be a quadratic field. For every prime number $p=11$ or $p>13$, if $E$ is a $\Q$-curve of degree $d$ prime to $p$ defined over $K$, the image of $\Prep$ is in the normaliser of a split Cartan subgroup of $\pglep$, then $E$ has potentially good reduction at every prime ideal of $\ok$.  
\end{prop}

\begin{proof}
 Let $P \in \xosdp (K)$ be the point associated to $E$. For any $t \in \T$ such that $t(1+w_p) = 0$, $h_t ^+$ is a $\Q$-rational morphism so $h_t ^+ (P) \in J_0 (p)_\Q (K)$. If we call $\s$ the automorphism of $K$, 
\[
 h_t ^+ (P)^\s= h_t ^+ (P ^\s) = h_t ^+ ( w_d . P ) =h_t ^+ (P)
\]
because $h \circ w_d = h$ by construction, so that $h_t ^+ (P)$ is $\Q$-rational.

Now, if we also suppose that $t$ cancels $\gamma_\Ical$ and $t = 1 \mod \Ical$ (see Lemma \ref{Nakayamadonneannulateurgammaical} and above for notations), then $h_t ^+ (P)=0$. Indeed, it is a torsion point of $J_0 (p)$ because the Eisenstein quotient has rank zero. Taking any prime $\gP$ of $\ok$ above $p$, by Proposition \ref{reduccompcartdeploye} we know that $P$ reduces modulo $\gP$ in the smooth part of $\xosdp_\Z$ (by $(a)$ and $(b)$), and that $(h_t {^+})_\Z (P_\gP)$ is 0 in the group of components $\Phi$ of $J_0 (p)_{\Ocal'} \times \overline{\F_\gP'}$ (by $(c)$). Therefore, $h_t ^+ (P) = 0$ as reduction from the cuspidal subgroup $C$ to $\Phi$ is injective (Proposition \ref{groupecomposantescasgeneral} $(b)$). We can now apply Mazur's method.

Suppose there is a prime ideal $\lambda$ of $\ok$ (above the prime number $\ell$) such that $E$ has potentially multiplicative reduction at $\lambda$. Then $P$ reduces at a cusp $c$ modulo $\lambda$, hence is in the smooth part of $\xosdp_\Z$ (this is obvious when $\lambda$ is not above $p$, and we just explained why it is true when $\lambda$ is above $p$). Actually, this cusp must be a cusp above $\infty_\Fl \in \xsp_\Fl$. Indeed, if $\ell=p$, this is what we just proved as the other cusps are not in the smooth part (Proposition \ref{reduccompcartdeploye}), and if $\ell \neq p$, simple computation shows that the image of a cusp $c$ of $\xospcardp_\Q$ not above $\infty$ or $0 \in \xspcar_\Q$ by $h$ is $2 \cl([0] - [\infty])$. For $t \in \T$ as above, as $h_t ^+ (P) = 0$ and the Zariski closure $C_\Z$ of $C$ in $J_0 (p)_\Z$ is étale outside 2 (eg by Raynaud's Theorem 3.3.3 \cite{Raynaud74}), if $\ell >2$ , it prevents $P$ from reducing at $c$ modulo $\lambda$ (unless $n=1$ or 2, impossible when $p=11$ or 
$p>13$), and for $\ell=2$, it is only possible when $n$ divides 4 which is also impossible for $p>23$. Now we know $P$ reduces modulo $\lambda$ at a cusp above $\infty_\Fl \in \xsp_\Z$. After applying an Atkin-Lehner involution $w_{d'}$, $d' |d$, we can suppose that $P$ reduces modulo $\lambda$ at $(\infty,\infty)_\Fl$. Take now $t \in \T \backslash \ell \T$ still satisfying the previous conditions by Lemma \ref{Nakayamadonneannulateurgammaical}, so we can apply Proposition \ref{propimmformellecasdeploye}. As $h_t ^+ (P) = 0$, there is no problem with the case $\ell=2$ here, and we obtain $P = (\infty,\infty)$, which is a contradiction. Therefore, $E$ has potentially good reduction at every prime ideal $\lambda$.
\end{proof}

\subsection{Nonsplit Cartan case}
\label{Nonsplit Cartan case}
In this subsection, we do not provide any qualitative improvement on the third section of \cite{Ellenberg04}, but a quantitative one. The next proposition is the algebraic part of section 3 of \emph{loc. cit.}, which uses Mazur's method and relies on the existence of a rank zero quotient on a twisted jacobian (using crucially the results of Kolyvagin and Logachev). Small characteristic issues can be easily ruled out for $p \geq 7$ in this case (see proof of Proposition 3.9 of \cite{Ellenberg04}).

\begin{prop}[\cite{Ellenberg04}, proof of Proposition 3.6]
\label{propEllnonsplitcartan}
Let $K$ be an imaginary quadratic number field of discriminant $-D_K$, $\chi_K$ the associated Dirichlet character and $d>1$ a squarefree positive integer.

 Let $p \geq 7$ be a prime number not dividing $dD_K$.
If there exists an eigenform $f \in S_2 ( \G_0 (p^2)) ^{\rm{new}}$ such that $w_p \cdot f = f$ and $L (f \otimes \chi_{K},1) \neq 0$, 
then for every strict $\Q$-curve $E$ of degree $d$ defined over $K$ such that the image of $\Prep$ is included in the normaliser of a nonsplit Cartan subgroup of $\pglep$, $E$ has potentially good reduction at every prime ideal of $\ok$.
\end{prop}

Thanks to this Proposition and using the same type of analytic estimates as in \cite{Ellenberg04}, we obtain in the Appendix the following result : 

\begin{prop}
 \label{BonnerednonsplitCartan}
Let $K$ be an imaginary quadratic field of discriminant $-D_K$. 

Let ${p> 50 D_K^{1/4} \log(D_K)}$ be a prime number not dividing $D_K$. If $E$ is a $\Q$-curve defined over $K$, of degree coprime with $p$, without complex multiplication and such that the image of $\Prep$ is in the normaliser of a nonsplit Cartan subgroup of $\pglep$, then $E$ has potentially good reduction at every prime ideal of $\ok$. 
\end{prop}

\section{Runge's method}
\label{Runge}
Runge's method can be used on a modular curve $X$ to bound the absolute logarithmic height of $j$-invariants of $S$-integral points on $X$, where $S$ is a set of places of the number field $K$ containing the infinite places, of cardinality smaller than the number of Galois orbits of cusps of $X$. In this article, our application of Runge's method will be a very simple case : $K$ is an imaginary quadratic field $X=X_0 (p)$ with $p$ a fixed prime number, and $S$ contains one single element : the euclidean norm on $\ok$. We define 
\[
X_0 (p) (\ok) := \left\{ P \in X_0 (p) (K) , j(P) \in \ok \right\}.
\]
Our goal is to explicitly bound the $j$-invariant of elements of $X_0 (p)(\ok)$. We denote by

$\Hcal$ the Poincaré upper half plane. 

$\Dcal$ the fundamental domain of $\Hcal$ for the action of $\slz$ bounded by the geodesic triangle $\{0,1,\infty \}$

$q_\tau := \exp ( 2 i \pi \tau)$ for any $\tau \in \Hcal$ (the $\tau$ index will be omitted if $\tau$ is obvious). 

$q_\tau ^r : = \exp ( 2 i \pi r \tau)$ for any rational number $r$.

Finally, for sake of precision, we denote $c_\infty$ (resp. $c_0$) the cusp of $X_0 (p)$ which is the image of $\infty$ (resp. 0) by the canonical projection $\pi : \Hcal \cup \P ^1 ( \Q) \ra X_0 (p)$. 
Notice that the Runge's method will be applied for $p$ dividing the degree of a $\Q$-curve, so our result has to hold for any prime $p$.

For general Runge's method, we need some knowledge about modular units (see \cite{BiluParent09} for a general exposition in the case of modular curves), but in our case we will just use one, defined as follows.

\begin{defi}
 Let $p$ be a fixed prime number. The holomorphic function $g$ on $\Hcal$ is defined for all $\tau \in \Hcal$ by
\[
 g (\tau) = \frac{ \Delta ( \tau) } {\Delta ( p . \tau )} = q_\tau^{1-p} \prod_{\substack{n=1 \\ (p,n)=1}}^{+ \infty} (1 - q_\tau ^n)^{24}.
\]
where $\Delta$ is the discriminant modular form on $\Hcal$.
This is the quotient of two modular forms of weight 12 on $\G_0(p)$, hence it defines a modular function $U$ on $X_0(p)$.
\end{defi}

The modular function $U$ benefits the following properties : 
\begin{prop}
\label{propsU}
 $(a)$ For every $\tau \in \Hcal$ : 
\[
g (-1/\t ) = p ^{12} g^{-1} (\t / p)=p  ^{12} q^{(p-1)/p} \prod_{n=1}^{\infty} \left( \frac{1 - q^n}{1- q^{n/p}} \right) ^{24}.
\]

$(b)$ For $w_p$ the Atkin-Lehner involution of $X_0(p)$, $U \circ w_p = p^{12} U^{-1}$.

$(c)$  The divisor of $U$ on $X_0(p)$ is supported by cusps, more precisely 
\[
\div (U) = (p-1) ([c_0] - [c_\infty]).
\]

$(d)$ The function $U$ is a $\Q$-rational function on $X_0 (p)$ which is integral over $\Z[j]$.
\end{prop}

\begin{proof}
 The assertion $(a)$ implies $(b)$ because for every $\tau \in \Hcal$ of image $P$ in $X_0(p)$, $g ( \tau) = U (P)$ and $-1/ (p \tau)$ has image $w_p ( P)$ by definition.
 To prove $(a)$, we only write that for every $\tau \in \Hcal$, by definition of $g$,
\[
g (-1/\t) = \frac{\Delta ( -1/ \t)}{\Delta (-p/ \t)} = \frac{\t ^{12} \Delta (\t)}{p^{-12} \t ^{12} \Delta (\t/p)} = p ^{12} g ^{-1} (\t/p),
\] 
because $\Delta$ is a modular form of weight 12 on $\slz$.
Next, the discriminant modular form does not cancel on $\Hcal$, therefore the divisor of $U$ is indeed supported by the two cusps $c_\infty$ and $c_0$. The $q$-expansion of $g$ at $\infty$ shows that the order of the pole of $U$ at $c_\infty$ is $(p-1)$, and the order at $c_0$ is necessarily the opposite, which proves $(c)$.
The modular function $U$ is $\Q$-rational on $X_0(p)$ as a quotient of two $\Q$-rational modular forms of weight 12 on $X_0(p)$. Only the integrality remains to be proved. Recall that
\begin{equation}
\label{decoupagesldeuxgamma0}
 \slz = \G_0 (p) \cup \bigcup_{k \in \Z} \G_0 (p) \cdot \begin{pmatrix} 0 & 1 \\ -1 & -k \end{pmatrix}
\end{equation}
Indeed, for every $\gamma \in \slz$, either $\gamma \cdot \infty = \gamma' \cdot \infty$ with $\gamma' \in \G_0 (p)$, and in this case $\gamma \in \G_0(p)$ because the stabiliser of the cusp $\infty$ in $\slz$ is contained in $\G_0(p)$, or $\gamma \cdot \infty = \gamma' \cdot 0$ with $\gamma' \in \G_0(p)$. The matrix $w = \begin{pmatrix} 0 & 1 \\ -1 & 0 \end{pmatrix}$ sends $\infty$ to 0, so we can write 
\[
 (\gamma ' w)^{-1} \gamma = \pm \begin{pmatrix} 1 & k \\ 0 & 1 \end{pmatrix}
\]
for some integer $k$, as $(\gamma' w)^{-1} \gamma \cdot \infty = \infty$. This proves \eqref{decoupagesldeuxgamma0}. From this equation, we know that for every $\gamma \in \slz$, the $q$-expansion of $g _{|\gamma}$ (that is, the image of $g$ by the usual right action of $\slz$ on functions on $\Hcal$) is a formal series in $q_ \tau ^{1/p}$ with algebraic integer coefficients. Indeed, if $\gamma \in  \G_0 (p) \cdot \begin{pmatrix} 0 & 1 \\ -1 & -k \end{pmatrix}$, 
\[
 g_{|\gamma} ( \tau) = g ( -1 / ( \tau + k)) =  p  ^{12} e^{2 i \pi (p-1) k/p} {q_\tau}^{(p-1)/p} \prod_{n=1}^{\infty} \left( \frac{1 - q_\tau^n}{1- e^{2 i \pi n k/p}q_\tau^{n/p}} \right) ^{24}
\]
by $(a)$, so this $q_\tau$-expansion has coefficients in $\Z [ e^{2 i \pi / p}] \subset \overline{\Z}$. Hence, from Lemma 2.1 of \cite{KubertLang} (Chapter 2.2), we know that $U$ is integral on $\Z[j]$.
\end{proof}

%
%
%
%
%
%

\begin{rem}
This proof is somewhat elementary, but the reader will notice that we only reproved some well-known results of the theory of modular units (which is far more general) for $U$. Actually, with the notations of (\cite{KubertLang}, Chapter 2), 
\[
 g := \prod_{a=1}^{p-1} g_{(\frac{a}{p},0)} ^{12p},
\]
and with results of the same chapter, we recover all the results of the previous Proposition except $(b)$. A consequence of $(b)$ is that $p^{12} U^{-1}$ is integral over $\Z [j]$, and it seems that the theory of modular units would only have predicted that $p ^{12p} U^{-1}$ is (see \cite{BiluParent09}).
\end{rem}

\begin{lem}
\label{encadrementUpointsok}
 For every $P \in Y_0 (p) ( \ok)$, $U(P)$ is a nonzero element of $\ok$ such that
\[
0 \leq \log |U(P)| \leq 12 \log p.
\]
\end{lem}

\begin{proof}
As $U$ is $\Q$-rational and integral over $\Z[j]$ (Proposition \ref{propsU} $(d)$), $U(P) \in \ok$ and is nonzero because $U$ does not cancel on $Y_0(p)$. The same thing is true for $w_p \cdot P$ : indeed, $w_p \cdot P \in Y_0 (p) ( \ok)$ because $w_p$ is $\Q$-rational and $w_p \cdot P$ represents an elliptic curve isogenous to the elliptic curve represented by $P$, so $j ( w_p \cdot P) \in \ok$ too. Therefore, $U \circ w_p ( P) = p ^{12} U^{-1} (P) \in \ok$. As $K$ is an imaginary quadratic field, for every nonzero element $\alpha \in \ok$, $\log | \alpha | \geq 0$, hence 
\[
0 \leq \log |U(P)| \leq 12 \log p.
\]
\end{proof}

We define the involution $w$ on $\Hcal$ by $w ( \tau) = -1/ \tau$ and the function $g_0$ on $\Hcal$ by $g_0: = g \circ w$.

The following lemma, that we call ``locating near cusps lemma'' allows us to reduce Runge's method to computation with the two functions $g$ and $g_0$.
\begin{lem}
 For every point $P \in Y_0 (p) (\C)$, there exists $\tau \in \Dcal + \Z$
such that $\tau$ or $-1/\tau$ is above $P$ by the canonical projection $\Hcal \ra Y_0 (p)(\C)$. In the first case, we say $P$ is near $c_\infty$, and then $j(P) = j(\tau)$ and $U (P) = g(\tau)$. In the second case, we say $P$ is near $c_0$, and then $j(P) = j(\tau)$ and $U(P) = g_0 ( \tau)$.
\end{lem}

\begin{proof}
 Let $P \in Y_0 (p) (\C)$. Choose a lift $\tau_0 \in \Hcal$ of $P$.  There exists $\beta \in \slz$ such that $\beta \cdot \tau_0=\tau_1 \in \Dcal$. This $\tau_1$ is not above $P$ anymore unless $\beta \in \G_0 (p)$ (in this case, choose $\tau = \tau_1$ in the lemma, and $P$ is near $c_\infty$). Suppose now $\beta \notin \G_0 (p)$. From \eqref{decoupagesldeuxgamma0}, we can write 
\[
 \beta^{-1}  = \gamma \cdot w \cdot \begin{pmatrix} 1 & k \\ 0 & 1 \end{pmatrix}
\]
for some $k \in \Z$ and $\gamma \in \G_0(p)$, $w= \begin{pmatrix} 0 & 1 \\ - 1 & 0 \end{pmatrix}$. Hence, $\tau = \begin{pmatrix} 1 & k \\ 0 & 1 \end{pmatrix} \cdot \tau_1 \in \Dcal + \Z$, and $w . \tau = \gamma^{-1} \cdot \tau_0$ is above $P$. In this case, we say $P$ is near $c_0$.
\end{proof}

We now need a lemma for precise estimates of $q$-expansions of $g$ and $g_0$.
\begin{lem}
\label{lemmajosommes}
 For every $r \in ]0,1[$ and every $q \in \C$ such that $|q| \leq r$,
\[
 \sum_{n=1} ^{+ \infty} | \log |1 - q^n|| \leq \frac{-\log(1-r)}{r(1-r)} |q|.
\]
For every $q \in \C$ such that $|q|<1$, 
\[
 \sum_{n=1}^{+ \infty} |\log |1 - q^n|| \leq  \frac{\pi ^2}{6\log |q^{-1}|}.
\]
\end{lem}

\begin{proof}
 The first inequality is a straightforward consequence of the triangular inequality and the maximum principle. The second one can be found in the proof of Lemma 3.5 of \cite{BiluParent09}.
\end{proof}

We obtain from this lemma nontrivial bounds on $g$ and $g_0$.
\begin{prop}
\label{propgqtau}
 For every $\tau \in \Dcal + \Z$, 
\begin{eqnarray*}
| \log |g(\t)| + (p-1) \log |q_\t| |  & \leq & 25 |q_\t|. \\
| \log |g_0 (\t)| - \frac{p-1}{p} \log |q_\t|| & \leq & \frac{4 \pi ^2 p}{\log |q_\t^{-1}|} + 12 \log(p).
\end{eqnarray*}
\end{prop}

\begin{proof}
 Using the $q$-expansion of $g$ and Lemma \ref{lemmajosommes} with $r=0.005$, we have
\[
 | \log |g_\infty (\t)| + (p-1) \log |q_\t| | = 24 \left| \sum_{\substack{n \geq 1 \\(p,n)=1}} \log |1 - q_\t ^n| \right| \leq - 24 \frac{\log (0.995)}{0.995 \cdot 0.005} |q_\t|\leq 25 |q_\tau|
\]
because $|q_\tau| \leq 0.005$ when $\tau \in \Dcal + \Z$.
 For the bound with $g_0$, we use the other inequality of Lemma \ref{lemmajosommes} with $q$-expansion of $g_0$ and obtain 
\[
 | \log |g_0 (\t)| - \frac{p-1}{p} \log |q_\t| - 12 \log(p)| = 24 \left| \sum_{\substack{n \geq 1 \\(p,n)=1}} \log |1 - q_\t ^{n / p}| \right| \leq \frac{4 \pi ^2}{\log |q_\t ^{-1/p}|} = \frac{4 \pi ^2 p}{\log |q_\t^{-1}|}.
\]

\end{proof}

Finally, we recall an inequality for the $j$-invariant (extracted from Corollary 2.2 $(iii)$ of \cite{BiluParent09}).
\begin{prop}
\label{propjinvariantqtau}
 For every $\tau \in \Dcal + \Z$, if $|j(\tau)| >3500$, then $\log |j(\tau)| \leq \log | q_\tau ^{-1}| + \log(2)$.
\end{prop}

We can now state our bound on the $j$-invariant.
\begin{prop}
\label{Rungeborne}
 Let $K$ be an imaginary quadratic field. For every prime number $p$ and every point $P \in Y_0 (p) (\ok)$, 
\[
 \log|j (P)| < 2 \pi \sqrt{p} + 6 \log(p) + 8.
\]
\end{prop}

\begin{proof}
 Let $\tau$ be a point of $\Dcal + \Z$ associated to $P$ by the ``locating near cusps lemma''.
 If  we have $\log|j (P)| < 2 \pi \sqrt{p}$, there is nothing to prove. If not, $|j ( \tau)| >3500$ hence $\log |j(\tau)| \leq \log | q_\tau ^{-1}| + \log(2)$ by Proposition \ref{propjinvariantqtau}. 
We now have to bound $|q_\tau ^{-1}|$. If $P$ is near $c_\infty$, we have $\log | g ( \tau)| =  \log |U(P)| \leq 12 \log p$ by Lemma \ref{encadrementUpointsok}. By Proposition \ref{propgqtau}, we obtain
\begin{equation}
\label{casPpresinfini} 
 \log |q_\t ^{-1}| \leq \frac{25 |q_\t| + 12 \log p}{p-1} \leq 2 \pi \sqrt{p} 
\end{equation}
after a little analysis (as $|q_\tau| \leq 0.005$ here).
If $P$ is near $c_0$, $\log | g_0 ( \tau)| =  \log |U(P)| \geq 0$ by Lemma \ref{encadrementUpointsok}. By Proposition \ref{propgqtau}, we obtain this time
\[
 \left( \frac{p-1}{p} \right) \log |q_\t ^{-1}| \leq \frac{4 \pi ^2 p}{\log |q_\t^{-1}|} + 12 \log(p),
\]
hence
\begin{equation}
\label{casPpreszero}
  \log |q_\t ^{-1}| \leq \frac{2 \pi p}{\sqrt{p-1}} + \frac{6 p \log(p)}{(p-1)} \leq 2 \pi \sqrt{p} + 6 \log(p) + 7
\end{equation}
because $p \geq 2$, after a small analysis on the remaining terms. In each case, \eqref{casPpresinfini} or \eqref{casPpreszero} gives us the result.
\end{proof}

\section{Isogeny theorems and end of the proof}
\label{theoremesdisogenie}

We now use isogeny theorems from \cite{GaudronRemond} to complete the proof and obtain completely explicit bounds.
These theorems use the notion of stable Faltings height $h_{\Fcal}$ (see \cite{GaudronRemond}, subsection 2.3 for details), but we recall that for every elliptic curve $E$ defined over a number field $K$, 
\begin{equation}
\label{lienFaltingsWeil}
 h_{\Fcal}(E) \leq \frac{1}{12} h(j(E)) + 2.38
\end{equation}
with $h$ the absolue logarithmic height of $j(E) \in K$, that is,
\[
 h (j(E)) = \frac{1}{[K : \Q]} \sum_{v  \in M} \max ( 0, \log ( |j(E)|_v))
\]
where $M$ is the set of places of $K$ : see Lemma 7.9 of \cite{GaudronRemond}. The different normalisations for Faltings height in \cite{GaudronRemond} have to be taken into account : with the notations of this article, $h(E) = h_\Fcal(E) + \log(\pi)/2$, so we obtain the constant $2.38$ instead of $2.95$, but this does not matter for the following, as we round up to 3.

The following result is proved in \cite{GaudronRemond} in part 7.3, but is not stated this way. We explain why after its statement.

\begin{prop}[\cite{GaudronRemond}, part 7.3]
 \label{thmisogenieCEmodifie}
Let $E$ be an elliptic curve without complex multiplication and $B$ an abelian surface both defined over the number field $K$. Let $\psi : B \ra E \times E$ be an isogeny defined over $K$. Suppose (hypothesis $(\ast)$) that for every embedding $\sigma : K \ra \C$, if $\Omega_{E,\sigma}$ and $\Omega_{B,\sigma}$ are the period lattices of $E$ and $B$ with respect to this embedding, $\mathrm{d} \psi (\Omega_{B,\sigma})$ (which is a sublattice of $\Omega_{E,\sigma} ^2$) contains an element $(\omega_1, \omega_2)$ of $\Omega_{E,\sigma}$ which is a $\Z$-basis of $\Omega_{E,\sigma}$. Then,
\[
 \deg ( \psi) \leq 10 ^7  [K : \Q]^2 ( \max \{ h_\Fcal (E), 985 \} + 4 \log [K : \Q] ) ^2.
\]
\end{prop}

\begin{proof}
The bound given here is exactly the bound of Theorem 1.4 of \cite{GaudronRemond}, because the computation is almost exactly the same. Consider some isogeny $\psi: B \ra E \times E$ satisfying hypothesis $(\ast)$ above. For every embedding $\sigma$, there is a canonical norm $\| \cdot \|_\sigma$ (coming from a principal polarization of $E$) on the tangent space $t_{E,\sigma}$, which contains $\Omega_{E,\sigma}$. We fix an embedding $\sigma_0$ such that there is a basis $(\omega_1,\omega_2)$ of $\Omega_{E,\sigma}$ which is minimal amongst all minimal bases for all possible period lattices $\Omega_{E,\sigma}$. This means that 
\[
 \| \omega_1 \|_{\sigma_0} = \max_{\sigma} \min_{\substack{\omega \in \Omega_{E,\sigma} \\ \omega \neq 0}} \| \omega \|_\sigma
\]
and $\omega_2 = \tau \omega_1$ with $\tau$ in the Siegel fundamental domain (so that $y = \im (\tau)$ is minimal amongst all choices of embeddings and minimal bases for these embeddings). This choice of $\sigma_0$ to minimise $y$ is the same as in part 7.3 of \cite{GaudronRemond}. We now identify all considered abelian varieties with their scalar extensions to $\C$ via this embedding $\sigma_0$, and therefore omit all further mention to the embeddings in the notation. We can compose $\psi$ by an isomorphism of $E \times E$ so that $d \psi ( \Omega_E)$ contains the basis $(\omega_1,\omega_2)$ previously chosen, because we assumed hypothesis $(\ast)$.
We then consider $A = E \times E \times B$ and a period $\omega=(\omega_1,\omega_2,\chi)$ of $\Omega_E ^2 \times \Omega_B$, with $\chi \in \Omega_B$ such that $\mathrm{d} \psi ( \chi) = (\omega_1, \omega_2)$. The minimal abelian subvariety $A_\omega$ of $A$ containing $\omega$ in its tangent space is then
\[
 A_\omega = \left\{ (\psi(z),z), z  \in  B \right\}.
\]
Indeed, the inclusion $A_\omega \subset \left\{ (\psi(z),z), z  \in  B \right\}$ is clear and the projection from $A_ \omega$ to $E \times E$ is a subvariety of $E \times E$ containing $(\omega_1,\omega_2)$ in its period lattice. As $E$ is an elliptic curve without complex multiplication, the endomorphism ring of $E \times E$ is $M_2 ( \Z)$, therefore no strict abelian subvariety of $E \times E$ contains $(\omega_1,\omega_2)$ in its tangent space. This proves that the dimension of $A_ \omega$ is at least 2, hence the equality above.
The abelian variety $A_\omega$ is canonically isomorphic to $B$ and the projection to $E \times E$ is an isogeny of degree $\Delta$. Everything is in place, and from now on, we can repeat the computation of part 7.3 of \cite{GaudronRemond} to obtain the bound of the proposition.
There are only two small differences to be noticed : first, the embedding $\sigma_0$ can be real or complex, and the bound can change if $\sigma_0$ is real as part 7.3 uses that $\sigma_0$ is complex to improve slightly on the bound. To avoid this issue, we consider $K' = K(i)$ and go back from the start with $K'$ instead of $K$. This sticks with the proof of part 7.3 as here also, an extension of degree at most 2 of $K$ was needed.
Second, the computation of the slopes in Lemma 7.6 of \cite{GaudronRemond} is slightly different, but using only that $B$ is isogenous to $E \times E$ with degree $\Delta$ will give the exact same bound. 
\end{proof}

For our results on $\Q$-curves, we will prove the following result that gives an explicit version of Serre's surjectivity theorem, as stated by Masser-Wüstholz in \cite{MasserWustholz93bis}.
\begin{thm}
\label{Serreexplicite}
 Let $E$ be an elliptic curve defined over the number field $K$ without complex multiplication. We define $\Bcal$ (resp. $\Ccal$) a set of prime numbers $p$ such that the image of $\rep$ is included in a Borel subgroup (resp. the normaliser of a Cartan subgroup, split or nonsplit) of $\GL(E_p)$. Then, the following inequality holds : 
\[
 \prod_{p \in \Bcal} p \prod_{q \in \Ccal} \frac{q^2}{4} \leq 10^7 [K : \Q] ^2 \left(  \max \{ h_\Fcal (E), 985 \} + 4 \log [K : \Q] + 4 |\Ccal| \log(2) \right) ^2.
\]
In particular, the Galois representation $\rep$ is surjective for every prime number 
\[
 p > 10^7 [K : \Q] ^2 \left(  \max \{ h_\Fcal (E), 985 \} + 4 \log [K : \Q] \right) ^2
\]
not dividing the discriminant of $K$.
 \end{thm}

First, we prove a technical lemma to help us prove our isogeny $\psi$ satisfies $(\ast)$.
\begin{lem}
 \label{lemmquadformdet}
Let $p$ be a prime number, $V$ be an $\Fp$-vector space of dimension 2 and $\underline{v}$ be a basis of $V$. Then, for $g \in \GL(V)$, the quadratic form
\[
 Q : x \longmapsto \det_{\underline{v}} ( x, g.x)
\]
is surjective in $\Fp$ if $g$ is semisimple and not an homothety.
\end{lem}

\begin{proof}
First, if $g$ is simple, $Q$ does not have any isotropic vector, and is therefore surjective (looking at its expression in an orthonormal basis).
In the other cases, as for every $x,y \in V$ and any bases $\underline{v}$ and $\underline{v'}$, 
\[
\det_{\underline{v'}} (x,y) =\det_{\underline{v'}}(\underline{v}) \cdot \det_{\underline{v}}(x,y)
\]
from the properties of determinant, two quadratic forms $Q$ built from two different choices of bases are proportional. Therefore, $Q$ is surjective if and only if it is surjective for one/any other choice of basis. We assumed $g$ is semisimple but not simple, therefore  there is a basis $\underline{v} =(v_1,v_2)$ in which $g$ is diagonal with distinct eigenvalues $\lambda$ and $\mu$.
In this basis, the expression of $Q$ is simply, for $x = x_1 v_1 + x_2 v_2$, 
\[
 Q(x) = (\mu - \lambda) x_1 x_2,
\]
hence $Q$ is surjective.
\end{proof}

We can now prove Theorem \ref{Serreexplicite}.

\begin{proof}
We begin with the first inequality and will explain afterwards how it implies the effective version of Serre's surjectivity theorem. Let us fix $K'$ an extension of $K$ of degree $2^{|\Ccal|}$ such that for every $q \in \Ccal$, the image of $\rho_{E,q}$ restricted to $\Gal ( \Kb / K')$ is included in a Cartan subgroup. 
Define $n_\Bcal= \prod_{p \in \Bcal} p$ and $n_\Ccal = \prod_{q \in \Ccal} q$. Then the proposition to prove is equivalent to
\[
 n_\Bcal n_\Ccal ^2 \leq 10^7 [K' : \Q] ^2 \left( ( \max \{ h_\Fcal (E), 985 \} + 4 \log [K' : \Q] \right) ^2.
\]
Therefore, we only need to find an abelian variety $B$ and an isogeny $\psi : B \ra E \times E$ both defined over $K'$, and such that $\psi$ has degree $n_\Bcal n_\Ccal ^2$, satisfying the hypothesis $(\ast)$ of Proposition \ref{thmisogenieCEmodifie}.
For every $p \in \Bcal$, we choose $C_p$ an $\Fp$-line fixed by $\rep$ and define $G_p = C_p \times E[p]$, which is a subgroup of $E[p] ^2$ of cardinality $p^3$. For every $q \in \Ccal$, we choose an element $g_q$ of the associated Cartan subgroup which is not an homothety (notice such an element is always semisimple). Then, we consider $ G_q := \{ (x, g_q \cdot x), x \in E[q] \} \subset E[q]^2$. 
This group is of cardinality $q^2$ and stable by the diagonal action of the Cartan subgroup (as Cartan subgroups are commutative), hence defined over $K'$ in $E \times E$ by hypothesis. We now consider 
\[
 G = \bigoplus_{p \in \Bcal} G_p \oplus \bigoplus_{q \in \Ccal} G_q \subset E[n_\Bcal n_\Ccal] ^2.
\]
This is a group of order $n_\Bcal ^3 n_\Ccal ^2$, defined over $K'$ by hypothesis. We consider the quotient abelian variety $B = (E \times E)/G$ and the quotient morphism $\varphi : E \times E \ra B$ with kernel $G$. As $G$ is contained in  $E[{n_\Bcal n_\Ccal}]^2$, there exists an isogeny $\psi :B \ra E \times E$ such that
\[
 \psi \circ \varphi = [n_\Bcal n_\Ccal],
\]
that is, the multiplication by $n_\Bcal n_\Ccal$ on $E \times E$. This isogeny is defined over $K'$, and by degree multiplicativity, we find $\deg ( \psi ) = n_\Bcal n_\Ccal^2$. Hence, we only have to prove now that $\psi$ satisfies $(\ast)$.

Let $\sigma : K' \ra \C$ be an embedding. We now prove hypothesis $(\ast)$ for $\psi$ and $\sigma$. For this embedding, $E$ is naturally identified with the quotient of its tangent space $t_{E,\sigma}$ by its period lattice $\Omega_{E,\sigma}$ (we then omit further reference to $\sigma$ here). If $\pi$ is the projection $t_E \times t_E \ra E \times E$, the lattice $\Omega = \pi^{-1} (G)$ of $\t_E ^2$ defines a quotient abelian variety $t_E \times t_E / \Omega$ that is isomorphic to $B$. With these embeddings, we have the commutative diagram
 \[
\xymatrix{t_E \times t_E \ar[r]^{\Id} \ar[d]_{\pi} & t_E \times t_E \ar[r]^{n_\Bcal n_\Ccal} \ar[d] & t_E \times t_E \ar[d]_{\pi} \\
 E \times E \ar[r]^{\varphi} & B \ar[r]^{\psi} & E \times E.
 }
\]
Therefore, it remains to prove that $\Omega ' = n_\Bcal n_\Ccal \Omega \subset \Omega_E \times \Omega_E$ contains a basis of $\Omega_E$. 
Choose a basis $(e_1,e_2)$ of $\Omega_E$.

Now, consider the image of $\Omega '$ in $(\Omega_E / p \Omega_E)^2$ for $p \in \Bcal \cup \Ccal$. Notice that after multiplication by $1/p$, this image identifies with a subgroup of $((\Omega_E /p) / (\Omega_E))^2$, that is, $E[p]^2$. From the definitions of $\Omega$ and $\Omega'$, the image of $\Omega'/p$ in $E[p] ^2$ is $((n_\Bcal n_\Ccal)/p) G_p = G_p$ (the prime-to-$p$ part of $\Omega$ is mapped to $\Omega_E ^2$ when multiplied by $n_\Bcal n_\Ccal/p$). A canonical $\Fp$-basis of $E[p] ^2$ is $\pi(e_1/p,e_2/p)$ and we now identify $E[p]$ to $\Fp^2$ with this choice of basis.

For $p \in \Bcal$, we choose a nonzero vector $(a,b) \in C_p$. Then, we fix $(c,d)$ such that $ad-bc = 1 \mod p $. Therefore, the vector $((a,b),(c,d))$ belongs to $G_p$ and is of determinant $1$ in the canonical basis  $\pi(e_1/p,e_2,p)$.

For $q \in \Ccal$, by Lemma \ref{lemmquadformdet}, we can choose $x \in E[q]$ such that $\det_{\pi(e_1/q,e_2,q)} (x,g_q \cdot x) = 1$. Therefore, the subgroup $G_q$ of $E[q] ^2$ contains an element of determinant 1 in the canonical basis $\pi(e_1/q,e_2,q)$.

We just proved that for every $p \in \Bcal$, there exists a matrix $\gamma_p = \matabcd \in \SL_2 ( \Fp)$ such that 
\[
 \matabcd \cdot \begin{pmatrix} \overline{e_1} \\ \overline{e_2} \end{pmatrix} \in \Omega' / (p \Omega_E)^2
 \] 
The specialisation morphism $\slz \ra \prod_{p \in \Bcal \cup \Ccal} \SL_2 ( \Z / p \Z)$ is known to be surjective (\cite{LangAlgebra}, Chapter XIII, Exercise 18)
 , therefore there exists $\gamma \in \slz$ such that 
 \[
 \gamma \cdot \begin{pmatrix} e_1 \\  e_2 \end{pmatrix} \in \Omega' + (n_\Bcal n_\Ccal) \Omega_E ^2 \subset \Omega',
 \]
 because $\Omega$ contains $\Omega_E ^2$. Such an element $\gamma \cdot (e_1,e_2)$ is by construction a basis of $\Omega_E$, therefore $\Omega'$ contains a basis of $\Omega_E$. This concludes the proof of the first inequality of the theorem.
 
 Now, this result implies the surjectivity theorem in the following way : take $p$ a prime number such that $\rep$ is not surjective, not dividing the discriminant of $K$. Then, the image of $\rep$ is included in a Borel subgroup, the normaliser of a Cartan subgroup, or an exceptional subgroup. In the Borel case, we obtain $p \leq 10^7 [K : \Q] ^2 \left(  \max \{ h_\Fcal (E), 985 \} + 4 \log [K : \Q] \right) ^2$. In the Cartan case, we obtain $p \leq 2 \cdot 10^{3.5} [K : \Q] \left(  \max \{ h_\Fcal (E), 985 \} + 4 \log [K : \Q] + 4 \log(2) \right)$. In the exceptional case, we have $p \leq 30 [K : \Q] + 1$ by Proposition \ref{exceptionalcase}. The maximum of these three bounds is obtained for the Borel case, and it gives us the effective surjectivity theorem.
\end{proof}

\begin{rem}
Some papers, such as \cite{Cojocaru05}, give other (nonexplicit) versions of Serre's surjectivity theorem, but based on the conductor $N_E$ of $E$.
Note however, that by \cite{Cojocaru05}, Theorem 3 (which assumes the degree conjecture...), our result implies that for $E$ defined over $\Q$, one would have surjectivity for $p \gg \log (N_E)^2$.
\end{rem}

This gives us our direct application for $\Q$-curves.

\begin{cor}
 Let $E$ be a $\Q$-curve without complex multiplication, of squarefree degree $d(E)$ over a quadratic field $K$.

If $\Prep$ is reducible for some prime $p$ not dividing $d$,
\[
 d(E)p \leq 10 ^7  [K : \Q]^2 ( \max \{ h_\Fcal (E), 985 \} + 4 \log [K : \Q] ) ^2
\]
If $\Prep$ has image in the normaliser of a Cartan subgroup for some prime $p$ not dividing $d$, 
\[
d(E)p^2 \leq 4 \cdot 10^7 [K : \Q]^2 (  \max \{ h_\Fcal (E), 985 \} + 4 \log (2[K : \Q])) ^2.
\]
\end{cor}
 
We can now state the main theorem in detail. 

\begin{thm}
Let $K$ be an imaginary quadratic field of discriminant $-D_K$. For every strict $\Q$-curve $E$ defined over $K$ with degree $d(E)$ and every prime number $p$ coprime with $d(E)$,
 
 $\bullet$ If $p \geq 2 \cdot 10^{13}$, the representation $\Prep$ is not included in a Borel subgroup.
 
 $\bullet$ If $p \geq 10^7$ the representation $\Prep$ is not included in the normaliser of a split Cartan subgroup.
 
 $\bullet$ If $p \geq \max ( 10^7, 50 D_K^{1/4} \log(D_K))$ and does not divide $D_K$, the representation $\Prep$ is not included in the normaliser of a nonsplit Cartan subgroup.
 
 $\bullet$ If $p \geq 67$, the representation $\Prep$ is not included in an exceptional subgroup.
\end{thm}

\begin{proof}
The exceptional case is solved by Proposition \ref{exceptionalcase}. For the other cases, we use Propositions \ref{bonnereductionpartoutBorel}, \ref{BornebonnereductioncasCartandeploye} and \ref{BonnerednonsplitCartan} to know that with the given bounds, if $\Prep$ is included in one of the three types of maximal subgroups, $j(E) \in \ok$. Consider now $d_0$ the smallest prime divisor of $d$. We use Runge's method in $X_0 (d_0) ( \ok)$ (Proposition \ref{Rungeborne}) to obtain 
\[
 \log |j(E)| \leq 2 \pi \sqrt{d_0} + 6 \log (d_0) + 8.
\]
As $j(E) \in \ok$ and $K$ is imaginary quadratic, ${h_\Fcal (E) \leq  (\log |j(E)|)/12 + 3}$ by \eqref{lienFaltingsWeil}. Computing with the previous Corollary and the Runge bound above, we finally obtain the theorem. 
\end{proof}

Some comments are in order about the different steps of the proof here. First, about Mazur's method, we suspect that it can be used for any central $\Q$-curve over any number field in Borel and split Cartan cases, giving bounds depending only on the degree of this field (but maybe growing exponentially in this degree). Indeed, both proofs in theses cases mainly rely on the fact that $X_0(d) \ra X(1)$ is unramified at $\infty$ and ramified at 0, but it is actually ramified at \emph{every} cusp of $X_0(d)$ different from $\infty$, so we could make use of all Atkin-Lehner involutions group and obtain a formal immersion satisfying the good conditions. Furthermore, all the theoretical work is done for the analysis of components group by Proposition \ref{groupecomposantescasgeneral}, although a precise equivalent of Lemma \ref{lemcalculgP} for a general number field is likely to be hard to write. About the nonsplit Cartan case, it is not completely clear for now what happens, but Ellenberg's trick might be applied 
for some (but not all) bigger number fields.

Regarding Runge's method, because for a central $\Q$-curve of degree $d$, with $d$ having $r$ prime factors, the number of rational cusps of $X_0 (d)$ is $2^r$, which is also the expected degree of the field of definition $K$ of $E$. Therefore, for Runge's condition to hold, we need $K$ not to be a totally real number field. The exponent on $d$ we should expect to bound $\log |j(E)|$ is not known for now (but it may be $1/2$ as in the quadratic case).

We would also like to point out that up until now, if we know about infinite families of quadratic $\Q$-curves of degrees 2,3,5,7 or 11 over quadratic fields (see for instance \cite{Hasegawa97} for parametrisations), examples on bigger fields turn out to be more rare. Of course, for reasons of genus, when the degree $d$ is too large, there is only a finite number of $\Q$-curves of degree $d$ but the limit has not been precisely computed yet.

\section{\texorpdfstring{Appendix : Analytic estimates of weighted sums of $L$-functions}{Appendix : Analytic estimates of weighted sums of L-functions}}

The aim of this appendix is to prove Proposition \ref{BonnerednonsplitCartan}. We fix an imaginary quadratic field $K$ of discriminant $-D$ and $\chi_K$ (often shortened in $\chi$) its quadratic character. Thanks to Proposition \ref{propEllnonsplitcartan} (we use the same notations), we only have to find, for every $p > 50 D^{1/4} \log(D)$ not dividing $D$, an eigenform $f \in S_2 ( \G_0 (p))^{new}$ such that $w_p \cdot f = f$ and $L( f \otimes \chi_K, 1) \neq 0$.

For every positive integers $m,N$ and every vector space $V \subset S_2(\G_0 (N))$, we denote by $a_m$ and $L_\chi$ the  linear forms defined on $V$ by
\[
 a_m : f \mt a_m (f) , \quad L_\chi :  f \mt L ( f \otimes \chi_K,1).
\]
We recall that $S_2 ( \G_0 (N))^+$ (resp. $S_2 ( \G_0 (N))^{-})$) is the space of modular forms of $S_2 ( \G_0 (N))$ such that $w_N \cdot f = f$ (resp. $w_N \cdot f = -f$). We also denote by $S_2 ( \G_0 (N))^{old}$ (resp. $S_2 ( \G_0 (N))^{new}$) the space of old (resp. new) modular forms of $S_2 ( \G_0 (N))$.
The scalar product of $a_m$ and $L_\chi$ on a vector space $V \subset S_2 ( \G_0 (N))$ is 
\[
 (a_m,L_\chi)_V : = \sum_{f \in \Fcal_V} \overline{a_m (f)} L ( f \otimes \chi_K,1),
\]
where $\Fcal_V$ is any Petersson-orthonormal basis of $V$.
This will be shortened in $(a_m,L_\chi)_N$ (resp. $(a_m,L_\chi)_N^{new}$, $(a_m,L_\chi)_N^{old}$, $(a_m,L_\chi)_N^{+,new}$, $(a_m,L_\chi)_N^{+}$, $(a_m,L_\chi)_N^{-}$) for $V= S_2 ( \G_0 (N))$ (resp. 

$S_2 ( \G_0 (N))^{new}$, $S_2 ( \G_0 (N))^{old}$, $S_2 ( \G_0 (N))^{+,new}$, $S_2 ( \G_0 (N))^{+}$, $S_2 ( \G_0 (N))^{-}$), and similarly with $(a_m,a_n)_V$. We will prove the following result which immediately implies Proposition \ref{BonnerednonsplitCartan}.
\begin{prop}
 \label{nonnulliteweightedsumL}
For every imaginary quadratic field $K$ with discriminant $-D$ and Dirichlet character $\chi$, and every prime number $p >50 D ^{1/4} \log(D)$ not dividing $D$, 
\[
 (a_1, L_\chi)_{p^2} ^{+,new} \neq 0.
\]
\end{prop}

For any positive integer $M$ and any eigenform $g \in S_2 ( \G_0 (M))$, the $L$-function of $g$ admits a meromorphic continuation over $\C$ and 
\begin{equation}
\label{lfuntwistintegrale}
 L(g,1) = \left\{ \begin{array}{lcc}
4 \pi \int_{1/\sqrt{M}}^{+ \infty} g (iu) du & {\rm{ if }} & g_{|w_M} = -g. \\
0 & {\rm{ if }} & g_{|w_M} = g
                  \end{array}
\right.
\end{equation}
Moreover, if $f$ is a modular form in $S_2 (\G_0(N))$ with $N$ prime to $D$, $f \otimes \chi_K \in S_2 ( \G_0 ( D^2 N))$ and
\begin{equation}
\label{eqfoncftwistchi}
 (f \otimes \chi)_{|w_{D^2 N}} = \chi ( -N) (f_{|w_N} \otimes \chi)
\end{equation}
(for more details, see (\cite{Bump}, $\mathsection$ $I.5$)).
Hence, for an eigenform $f \in S_2 ( \G_0 (p^2))^{+}$, the sign of the functional equation of $L(f \otimes \chi)$  is $- \chi ( - p ^2)= 1$ because $\chi(-1)=-1$ : this is where we need $K$ to be imaginary to ensure that the twisting actually changes the sign of the functional equation for $L$. In the real quadratic case, we expect from Birch and Swinnerton-Dyer conjecture that there is no nontrivial rank zero quotient of the jacobian of $X_ {\textrm{nonsplit}} (p^2)$, because $J (X_ {\textrm{nonsplit}} (p^2))$ is isogenous over $\Q$ to $J_0 ({p^2})^{+,new}$ \cite{DeSmitEdixhoven}. In our situation, twisting reverts the sign of the functional equation, so $(a_1, L_\chi)_ {p^2}^{-}= 0$ and
\begin{equation}
\label{formulea1chip2old}
(a_1,L_\chi)_{p^2}^{new}=(a_1,L_\chi)_{p ^2} ^{+,new} = (a_1,L_\chi) _{p^2} - \frac{p}{p^2-1} (a_1,L_\chi)_p + \frac{\chi(p)}{p^2-1} (a_p,L_\chi)_p
\end{equation}
from Lemma 3.12 of \cite{Ellenberg04} and its proof.
The idea is that $(a_1, L_\chi)_ {p^2}$ is close to $4 \pi$ in modulus when $p$ is large enough compared to $D$, whereas the remaining term is close to 0.
From \eqref{formulea1chip2old}, we only need to estimate $(a_m, L_\chi)_N$ for $m=1,N=p^2$ (case (1)), $m=1, N=p$ (case (2)) and $m=N=p$ (case (3)) with $p \geq 7$ prime not dividing $D$ (the reader should mind throughout the estimates).

For an estimate of $(a_m,L_\chi)_N$, we will use Petersson's trace formula (\cite{IwaniecKowalski}, Proposition 14.5)
in a restricted version of Akbary (\cite{Akbary}, Theorem 3). For this, we define for every integers $m,n,c$ the Kloosterman sum
\[
 S(m,n;c) := \sum_{v \in (\Z/c\Z) ^*} e^{ 2 i \pi (m v + n \overline{v})/c} \,
\]
where for every $v \in (\Z / c \Z)^*$, $\overline{v}$ is the inverse of $v$.

Next proposition is a reformulation of Akbary's trace formula in weight 2 case.
\begin{prop} Let $m,n,N$ be three positive integers and $\varepsilon = \pm 1$. We have
\begin{eqnarray*}
 \frac{1}{2 \pi \sqrt{mn}} (a_m,a_n)_N ^+ &  = \delta_{mn} & \! \! \! \! \!  - 2 \pi \sum_{\substack{c >0 \\ N |c}} c^{-1} S(m,n ; c) J_1 \left( \frac{4 \pi \sqrt{mn}}{c} \right) \\ & & \! \! \! \! \! - \frac{2 \pi \varepsilon}{\sqrt{N}} \sum_{\substack{d > 0 \\ (d,N)= 1}} d^{-1} S(n,mN ^{\phi(d)-1}  ; d) J_1 \left( \frac{4 \pi \sqrt{mn}}{d \sqrt{N}} \right),
\end{eqnarray*}
where $\delta$ is the Kronecker symbol, $J_1$ is the Bessel function of the first kind of order 1, $(d,N)$ is the greatest common divisor of $d$ and $N$ and $\phi$ is the Euler totient function.
\end{prop}

Notice that summing Akbary's formula for $\varepsilon=1$ and $-1$ gives Petersson's trace formula.
As a normalized eigenform $f \in S_2 ( \G_0 (N))$ contributes for zero in $(a_m,L_\chi)$ if $f_{|w_N} = \chi ( -N) f$ by \eqref{eqfoncftwistchi}, $(a_m,L_\chi)_N = (a_m,L_\chi)_N ^\varepsilon$ with $\varepsilon = - \chi ( -N) = \chi(N)$. For every $f \in S_2 ( \G_0 (N))^{\varepsilon}$, we have from \eqref{lfuntwistintegrale}
\begin{equation}
\label{equationlfuntwist}
 L_\chi (f) =4 \pi \int_{1/ (D \sqrt{N})}^ {+ \infty} (f \otimes \chi) ( iu ) du  = 2 \sum_{n=1} ^{+ \infty} \frac{a_n (f) \chi(n)}{n} e^{- \frac{2 \pi n}{D \sqrt{N}}} 
\end{equation}
For convenience,  we use throughout this section the notation 
\[
 x = \frac{2 \pi }{D \sqrt{N}}.
\]
From \eqref{equationlfuntwist}, we get  $(a_m,L \chi)_N = 2 \sum_{n=1} ^{+ \infty} \frac{\chi(n)}{n} e^{-nx} (a_m,a_n)_N ^{\varepsilon}$ hence by Akbary's formula
\begin{equation}
\label{formuleamlchinamchibmchi}
(a_m,L \chi)_N = 4 \pi \chi (m) e ^{ - m x} - 8 \pi ^2 \sqrt{m} \left(  A(m,\chi,N) + \frac{ \varepsilon}{\sqrt{N}} B(m,\chi,N) \right)
\end{equation}
with
\begin{eqnarray*}
A(m,\chi,N) & := &  \sum_{n=1} ^{+ \infty} \frac{\chi(n)}{\sqrt{n}} e ^{- n x} \sum_{\substack{c >0 \\ N |c}} c^{-1} S(m,n ; c) J_1 \left( \frac{4 \pi \sqrt{mn}}{c} \right) \\
B(m, \chi,N) & := & \sum_{n=1} ^{+ \infty} \frac{\chi(n)}{\sqrt{n}}  e ^{- n x} \sum_{\substack{d > 0 \\ (d,N)= 1}} d^{-1} S(n,m N ^{\phi(d)-1} ; d) J_1 \left( \frac{4 \pi \sqrt{mn}}{d \sqrt{N}} \right)
\end{eqnarray*}
This is where our approach starts to differ from \cite{Ellenberg05} : this exact formula allows better results in general.
It can be readily checked with help of the Weil bounds (Proposition \ref{BornesWeilKloosterman}) and the classical bounds for Bessel functions that these double sums converge absolutely, so we can switch the terms. We obtain 
\begin{equation}
 \label{expressionAmchiN}
  \! \! \! \! \! \! \! \! \! \! \! \! \! \! \! \! \! \! \! \! \! \! \! \! \! \! \! \! \! \!A(m,\chi,N)  =  \sum_{\substack{c >0 \\ N |c}} \frac{\Scal_A (c)}{c}  \quad {\textrm{with}} \quad \Scal_A(c)  =   \sum_{n=1}^{+ \infty} \frac{\chi(n)}{\sqrt{n}} S (m,n ; c) J_1 \left( \frac{4 \pi \sqrt{mn}}{c} \right) e^{- n x}.
\end{equation}
\begin{equation}
\label{expressionBmchiN}
 B(m, \chi,N)  =  \! \! \! \! \sum_{\substack{d >0 \\ (d,N)=1}} \! \!  \! \! \! \frac{\Scal_B(d)}{d}\quad {\textrm{with}} \quad  \Scal_B (d)  =  \sum_{n=1}^{+ \infty} \frac{\chi(n)}{\sqrt{n}} S (n,m N^{\phi(d) -1} ;d) J_1 \left( \frac{4 \pi \sqrt{mn}}{d \sqrt{N}} \right) e^{- n x}.
\end{equation}

These sums are the ones we are going to bound in two different ways, for every $N|c$ and every $d$ prime to $N$. 
Let us recall the Weil bounds on Kloosterman sums.
\begin{prop}
\label{BornesWeilKloosterman}
 For every positive integers $m,n,c$, we have 
\[
| S (m,n ;c) |\leq (m,n,c) ^{1/2} \tau(c) \sqrt{c}
\]
with $\tau(c)$ the number of positive divisors of $c$. Furthermore, if for an odd prime $p$, we have $c = p^\alpha c'$ with $(p,c')=1$ and $\alpha \geq 1$,
\[
 | S (m,n ;c) | \leq 2 \tau (c') (m,n,c)^{1/2} \sqrt{c}
\]
and the latter bound is replaced by $ \tau(c') (m,n,c)^{1/2} \sqrt{c'}$ when $p$ divides $m$ but not $n$ (or the reverse) and $\tau(c/p) (m,n,c)^{1/2} \sqrt{c}$ when $p$ divides $m$ and $n$.
\end{prop}

\begin{proof}
 By multiplicativity of Kloosterman sums, this only needs to be checked for $c=p^\alpha$. If $p|nm$, this boils down to Ramanujan sums (\cite{IwaniecKowalski}, (3.2) and (3.3)). If $\alpha \geq 2$, there is an elementary proof (\cite{IwaniecKowalski}, Corollary 11.12), and if $\alpha=1$ and $(p,mn)=1$, this is a famous result of Weil (\cite{IwaniecKowalski}, Theorem 11.11).
\end{proof}

\begin{rem}
What we gain from the second bound is $2 \tau (c')$ instead of $\tau(c)$, which will be an advantage
for the case $(1)$ ($m=1$, $N=p^2$) and the third and fourth bounds will be of use in case $(3)$ ($m=N=p$) for the next proposition.
\end{rem}

\begin{prop}
For every $N|c$ and every $d$ prime to $N$,
\[
|\Scal_A(c)|  \leq  \frac{2 D \sqrt{N} \tau (c/N)}{\sqrt{c}}, \qquad 
|\Scal_B(d)|  \leq  \frac{D \sqrt{m} \tau(d)}{\sqrt{d}}
\]
in cases $(1)$, $(2)$ and $(3)$. These bounds will be called ``Weil-induced bounds''.
\end{prop}

\begin{proof}
For every $x \in \R$, we have $|J_1(x)| \leq x/2$. By triangular inequality, this gives
\[
|\Scal_A (c)|  \leq  \frac{2 \pi \sqrt{m}}{c} \sum_{n=1}^{+ \infty} |S(m,n;c)| e^{-nx} ,\qquad
|\Scal_B(d)|  \leq  \frac{2 \pi\sqrt{m}}{d \sqrt{N}} \sum_{n=1}^{+ \infty} |S(n,mN^{\phi(d)-1};d)| e^{-nx}.
\]
In cases $(1)$ and $(2)$, $m=1$ so $|S(m,n;c)| \leq 2 \tau(c/N) \sqrt{c}$ and $|S(n,mN^{\phi(d)-1} ; d)| \leq \tau(d) \sqrt{d}$ by Weil bounds, as $(d,N)=1$. But 
\[
\sum_{n=1}^{+ \infty} e^{-nx} = \frac{1}{e^x -1} \leq \frac{1}{x} =  \frac{D \sqrt{N}}{2 \pi}
\]
so we get the upper bounds for $\Scal_A$ and $\Scal_B$. In case $(3)$, $\Scal_B$ gives the upper bound by the same process as $(m,d)=1$. For $\Scal_A$, we use the third and fourth bounds of previous proposition and get
\begin{eqnarray*}
\Scal_A (c) & \leq & \frac{2 \pi \sqrt{p}}{c} \sum_{(n,p)=1} \tau \! \left(\frac{c}{p} \right) \sqrt{\frac{c}{p}} e^{-nx} + \frac{2 \pi \sqrt{p}}{c} \sum_{n=1}^{+ \infty} \tau \left(\frac{c}{p} \right) \! \sqrt{pc} \, e^{-pnx} \\
& \leq & \frac{2 \pi \tau(c/p)}{(e^x-1) \sqrt{c}} + \frac{ 2 \pi p \tau(c/p)}{(e^{px}-1) \sqrt{c}} \leq \frac{2 D \sqrt{p} \tau(c/p)}{\sqrt{c}}\\  
\end{eqnarray*}
with the same estimates.
\end{proof}
We will now obtain other bounds for $|\Scal_A|$ and $|\Scal_B|$ which will happen to be really sharper for ``small'' $c$ and $d$. They rely on  an Abel transform on the sums definining $\Scal_A$ and $\Scal_B$. For this, we use the following lemma which are analogues of Gauss sums and Polya-Vinogradov inequality for the twisting terms of \eqref{expressionAmchiN} and \eqref{expressionBmchiN}.

\begin{lem}
\label{BornessommesKloostermanDirichlet}
Let $c \neq D$ and $m$ be three fixed positive integers and $F$ the least common multiple of $c$ and $D$.
Let $\chi$ be a quadratic Dirichlet character of conductor $D$.
Then for every integer $\alpha$, 
\[
\left| \sum_{n=0}^{F-1} \chi(n) S(m,n;c) e^{2 i \pi n \alpha /F} \right| \leq c \sqrt{D}
\]
 and the sum is zero when $(\alpha,F/(c,D)) \neq 1$.
 With the same notations,  we have
 \[
 \sup_{K,K'  \in \N} \left| \sum_{n=K}^{K'} \chi(n) S(m,n;c) \right| \leq \frac{4 c\sqrt{D}}{\pi^2}( \log(Dc) + 1.5).
 \]
\end{lem}

\begin{rem}
 For $c=1$ (no Kloosterman sum), this is one of the versions of Polya-Vinogradov inequality for Dirichlet characters (weaker than \cite{Pomerance11}). For
$D=1$ (no Dirichlet character), this is an analogous inequality for Kloosterman sums that we could not find in the literature but should exist in sharper versions.
\end{rem}

\begin{proof}
We define $c'=c/(c,D)$ and $D'=D/(c,D)$. By definition of Kloosterman sums,
\begin{eqnarray}
\label{ouvertureKloosterman}
\sum_{n=0}^{F-1} \chi(n) S(m,n;c) e^{2 i \pi n \alpha /F} & = & \sum_{v \in (\Z/c\Z)^*} e^{2 i \pi m \overline{v}/c} \sum_{n=0}^{F-1} \chi(n) e^{2 i \pi n ( v/c + \alpha/F)} \\
& =& \sum_{v \in (\Z/c\Z)^*}  e^{2 i \pi m \overline{v}/c} \left( \sum_{n'=0}^{D-1} \chi(n') \theta^{n'} \right) \left( \sum_{\ell=0}^{c'-1} \theta^{\ell D} \right)
\end{eqnarray}
with $\theta = \exp ( 2 i \pi (v/c+ \alpha/F))$, because $\chi$ only depends on $n \! \mod D$. As $\theta^D$ is a $c'$-th root of unity, the  right term of this equality is nonzero if and only if $\theta ^D=1$, that is if and only if
\[
F | (D'v + \alpha)D \Llra c' | D'v + \alpha \Llra v = -(D')^{-1} \alpha \! \mod c'.
\]
Let $I_\alpha$ be the set of the invertible $v \in (\Z/c \Z)^*$ congruent to $(D')^{-1} \alpha \! \mod c'$. 
Notice first that if $(\alpha,c') \neq 1$, $I_\alpha = \emptyset$, hence the whole sum is zero.
If $(\alpha,c')=1$, $I_\alpha \neq \emptyset$ and we have
\[
\sum_{n=0}^{F-1} \chi(n) S(m,kn;c) e^{2 i \pi n \alpha /F} = c' \sum_{v \in I_\alpha} e^{2 i \pi m \overline{v}/c} \sum_{n'=0}^{D-1} \chi(n') \theta^{n'}.
\]
The inner sum is a Gauss sum on $\chi$ as $\theta$ is a $D$-th root of unit. More precisely, if we define $G(\chi) = \sum_{n=0}^{D-1} \chi(n) e^{2 i n \pi/D}$, by the usual properties of Gauss sums, $|G(\chi)|= \sqrt{D}$ and 
\[
\sum_{n=0}^{F-1} \chi(n) S(m,n;c) e^{2 i \pi n \alpha /F} = c' \sum_{v \in I_\alpha} e^{2 i \pi m \overline{v}/c} \chi \left(\frac{D'v + \alpha}{c'} \right) G(\chi).
\]

Now, if $(D', \alpha) \neq 1$, $\chi((D'v + \alpha)/c')= 0$ for all $v \in I_\alpha$ so the sum is zero. In the general case, the cardinality of $I_\alpha$ is at most $(c,D)$, so 
\[ 
\left|
\sum_{n=0}^{F-1} \chi(n) S(m,n;c) e^{2 i \pi n \alpha /F}\right| \leq c' (c,D) \sqrt{D} = c \sqrt{D} .
\]
We now prove the second inequality, using Polya-Vinogradov approach. Take $K,K'$ any integers. By discrete Fourier transform, as $\chi(n) S(m,n;c)$ is $F$-periodic in $n$, we have
\[
 \sum_{n=K}^{K'} \chi(n) S (m, n,c)  = \frac{1}{F}  \cdot \sum_{\gamma=0} ^{F-1} \left[  \sum_{\beta=0}^{F-1} \chi(\beta) S(m,\beta;c) e^{2 i \pi  \gamma \beta /F}  \sum_{n=K}^{K'}  e^{-2 i \pi \gamma n /F}  \right].
 \]
 The interest of the previous step of the lemma is now clear : we have a geometric sum on $n$ that is easy to bound on the right. Notice that the sum on $\beta$ at $\gamma = 0$ is zero from the previous results because $(0,c'D') \neq 1$. We obtain
\[
 \left| \sum_{n=K}^{K'} \chi(n) S (m, n,c) \right|   \leq  \frac{1}{F} \cdot \sum_{\gamma=1}^{F-1} c \sqrt{D} \left| \frac{1 - e^{- 2 i \pi \gamma (K'-K+1)/F}}{1-e^{-2 i \pi \gamma/F}} \right| 
  \leq  \frac{c \sqrt{D}}{F} \cdot \sum_{\gamma=1}^{F-1} \frac{|\sin ( \pi \gamma(K'-K+1) /F)|}{\sin ( \pi \gamma /F)}.
 \]
Hence, it only remains to prove that for every integers $F\geq 1$ and $K$,
\begin{equation}
 \label{majosommetrigosin}
S_{K,F} := \sum_{\gamma=1}^{F-1} \frac{|\sin ( \pi \gamma K /F)|}{\sin ( \pi \gamma /F)} \leq \frac{4F}{\pi^2} \cdot \left( \log(F) + 1.5 \right)
\end{equation}
From Lemma 2 and proof of Lemma 3 of \cite{Pomerance11}, we know that
for every $n \in \N$ and every $x \in \R$,
\begin{equation}
\label{lemmecosPomerance}
\sum_{j=1}^n \frac{ \cos(jx)}{j} > - \log(2) - \frac{2}{n},
\end{equation}
\begin{equation}
\label{lemmeinversesinPomerance}
A_{K,F} := \sum_{\gamma=1}^{F-1} \frac{1}{ \sin ( \pi \gamma/F)} \leq \frac{ 2 F}{\pi} \left( \log(F) + 0.13 \right).
\end{equation}
Using the Fourier series expansion of $|\sin \theta|$, we have
\begin{eqnarray}
S_{K,F}& = & \frac{2}{\pi} \sum_{\gamma=1}^{F-1} \frac{1}{ \sin ( \pi \gamma/F)} - \frac{4}{\pi} \sum_{m=1}^{+ \infty} \frac{1}{ 4 m^2 -1} \left( \sum_{\gamma=1}^{F-1} \frac{\cos ( 2 \pi m K \gamma /F)}{\sin ( \pi \gamma/F)} \right) \nonumber \\
\label{SKFAKFBKF}
S_{K,F} & = &  \frac{2}{\pi} A_{K,F} - \frac{4}{\pi} \sum_{m=1}^{+ \infty} \frac{B_{m,K,F}}{4 m^2 -1}.
\end{eqnarray}
The bound for $A_{K,F}$ is given by $\eqref{lemmeinversesinPomerance}$, to there is only $B_{m,K,F}$ left to study. 
Suppose $F$ is odd. For every $x \in [0,\pi/2]$, $\sin(x) = x - \varepsilon_x x^3/6 \geq 0$ with $\varepsilon_x \in [0,1]$, so we have
\[
\left| \frac{1}{\sin(x)} - \frac{1}{x} \right| \leq \frac{x}{6 - x^2}.
\]
From this, we get
\begin{eqnarray*}
B_{m,K,F} & = & 2 \sum_{\gamma=1}^{(F-1)/2} \frac{\cos ( 2 \pi m K \gamma /F)}{\sin ( \pi \gamma/F)} \\
& \geq & 2 \sum_{\gamma=1}^{(F-1)/2} \frac{\cos ( 2 \pi m K \gamma /F)}{\pi \gamma/F} - 2 \sum_{\gamma=1}^{(F-1)/2} \frac{ \pi \gamma /F}{6 - (\pi \gamma/F)^2} \\
& \geq &  -  \frac{2F}{\pi} (\log(2) + 4/(F-1)) - \frac{2F}{\pi} \int_0^{\pi/2} \frac{u du}{6 - u^2} \\
& \geq &  -  \frac{2F}{\pi} (0.96 + 4/(F-1)).
\end{eqnarray*}
Here, we use $\eqref{lemmecosPomerance}$ for the first sum and a careful sum-integral comparison for the second.
Finally, as $\sum_{m=1}^{+ \infty} 1/(4 m^2 -1) = 1/2$, we obtain from $\eqref{lemmeinversesinPomerance}$ and $\eqref{SKFAKFBKF}$ the inequality
\[
S_{K,F}  \leq  \frac{4F}{\pi^2} ( \log(F) + 1.09 + 4/(F-1)),
\]
which gives the result when $F \geq 11$, and we complete the proof by computation for $F \leq 10$. The proof works the same way for $F$ even, except we have to take aside the term $\gamma=F/2$.
\end{proof}

These bounds on partial sums will give us new bounds on $|\Scal_A|$ and $|\Scal_B|$ that we write in the next proposition.
\begin{prop}
For every integers $N|c$ and $(d,N)=1$ different from $D$, we have
 \[
 |\Scal_A(c)| \leq 6 \sqrt{Dm} (\log(Dc) + 1.5)  \qquad |\Scal_B(d)| \leq \frac{6\sqrt{Dm} (\log(Dd) + 1.5)}{\sqrt{N}}.
 \]
\end{prop}

\begin{proof}
For every integer $n \geq 0$, define
\[
A_n  :=  \sum_{k=1}^n \chi(k) S(m,k;c) , \qquad 
B_n := \sum_{k=1}^n \chi(k) S(k,mN^{\phi(d)-1} ;d).
\]
so that
\[
\Scal_A(c) = \sum_{n=1}^{+ \infty} (A_n-A_{n-1}) \frac{ 4 \pi \sqrt{m} f_A(n)}{c}, \qquad \Scal_B(d) = \sum_{n=1}^{+ \infty} (B_n-B_{n-1}) \frac{ 4 \pi \sqrt{m} f_B(n)}{d \sqrt{N}}
\]
with
\[
f_A(y) = \frac{c J_1 \left( \frac{4 \pi \sqrt{my}}{c} \right)}{4 \pi \sqrt{my}} e^{-yx}, \qquad f_B(y) = \frac{d \sqrt{N} J_1 \left( \frac{4 \pi \sqrt{my}}{d \sqrt{N}} \right)}{4 \pi \sqrt{my}} e^{-yx}.
\]
Therefore, with the Abel transform, we have
\[
|\Scal_A(c)|  \leq  \frac{4 \pi \sqrt{m}}{c} \cdot \sum_{n=1}^{+ \infty} |A_n| \cdot |f_A(n) - f_A(n+1)|  \qquad |\Scal_B(d)|  \leq  \frac{4 \pi \sqrt{m}}{d \sqrt{N}} \cdot \sum_{n=1}^{+ \infty} |B_n| \cdot |f_B(n) - f_B(n+1)| \\
\]
\[
|\Scal_A(c)|  \leq \frac{16\sqrt{Dm}}{\pi} \operatorname{Totvar} (f_A) (\log(Dc) + 1.5)  \qquad |\Scal_B(d)|  \leq \frac{16 \sqrt{Dm}}{\pi \sqrt{N}}  \operatorname{Totvar} (f_B) (\log(Dd) + 1.5)
\]
from Proposition \ref{BornessommesKloostermanDirichlet}, with $\operatorname{Totvar} (f_A)$ and $\operatorname{Totvar} (f_B)$ the total variations of $f_A$ and $f_B$ on $[0,+ \infty[$. It is clear from their expression that these two total variations are both bounded by the total variation of $J_1(x)/x$ on $[0,+ \infty[$, which is equal to
\[
\int_{0}^{+ \infty} \left| \frac{J_2(x)}{x} \right| dx \leq 1.1
\]
Rounding up $16/\pi \cdot 1.1$ to 6, we obtain the desired bounds.
\end{proof}

For the next estimates, we recall some standard bounds in the following lemma.
\begin{lem}
 For every integer $\lambda \geq 1$, we have
\[
 \sum_{n=1}^{\lambda} \frac{1}{n} \leq \log ( \lambda)+1, \qquad \sum_{n=1}^{\lambda} \frac{\log(n)}{n} \leq \frac{\log(\lambda)^2}{2}, \qquad \sum_{n \geq \lambda} \frac{\tau(n)}{n^{3/2}} \leq \frac{2 \log(\lambda) + 7}{\sqrt{\lambda}}.
\]
with $\tau(n)$ the number of positive divisors of $n$.
\end{lem}

\begin{proof}
 The first two bounds are given by sum-integral comparison. For the third one, we have
\begin{eqnarray*}
 \sum_{n \geq \lambda} \frac{\tau(n)}{n^{3/2}} & = & \sum_{k,\ell=1}^{+ \infty} \frac{\mathbf{1}_{k \ell \geq \lambda}}{(k\ell)^{3/2}} \\
 & = & \sum_{k=1}^{\lceil \lambda/2 \rceil -1} \sum_{\ell= \lceil \lambda/k \rceil}^{+ \infty} \frac{1} { (k \ell)^{3/2}} + \sum_{k= \lceil \lambda/2 \rceil}^{\lambda-1} \sum_{\ell=2}^{+ \infty}  \frac{1} { (k \ell)^{3/2}} + \sum_{k=\lambda+1}^{+ \infty}  \sum_{\ell=1}^{+ \infty} \frac{1}{(m \ell)^{3/2}}
\end{eqnarray*}
and from this cutting-up, we get the wanted bound after a careful computation (leaving small $\lambda$ aside, but we can check the bound for them afterwards).
\end{proof}

\begin{rem}
The bound given here is the sharpest with integral coefficients : for $\lambda=1$, the sum is about $6.8$, which is an explanation for the fact that we need to be very careful throughout the computation to obtain the coefficient 7 above.
\end{rem}

We can finally make the estimates for $A(m,\chi,N)$ and $B(m,\chi,N)$.
\begin{prop}
\label{bornesamchiNabel}
With the same notations as before, we have
\begin{eqnarray*}
|A(m,\chi,N)| & \leq & \min \left(\frac{14D}{N},\frac{\sqrt{Dm}}{N} \left(9 \log^2(D) +6 \log(D) \log(N) \right) \right) \\
|B(m,\chi,N)| & \leq & \min \left(7 D \sqrt{m},\frac{\sqrt{Dm}}{\sqrt{N}}\left(9 \log^2(D) + 12 \log(D) \log(N) + 6 \log^2 (N) \right) + \frac{\tau(D) \sqrt{m}}{\sqrt{D}} \right) 
\end{eqnarray*}
\end{prop}

Now, we finish the proof of Proposition \ref{nonnulliteweightedsumL}.
\begin{proof}
Let $\lambda \geq 1$ be a parameter. To bound $\Scal_A (Nc)$, we will choose the "Abel-transform bound" for $c < \lambda$ and the Weil-induced bound for $c > \lambda$. This gives us 
\begin{eqnarray*}
|A(m,\chi,N)| & = & \left| \sum_{c>0} \frac{\Scal_A (Nc)}{Nc} \right|\leq 6 \sqrt{Dm} \left( \sum_{c=1}^{\lfloor \lambda \rfloor} \frac{ \log(DNc) + 1.5}{Nc} \right)
+ 2 D \sqrt{N} \left( \sum_{c > \lceil \lambda \rceil} \frac{\tau (c)}{(Nc)^{3/2}}  \right) \\
\\
 & \leq & \frac{6 \sqrt{Dm}}{N} \left((\log(DN)+1.5)(1+ \log(\lambda)) + \frac{\log(\lambda)^2}{2} \right) + \frac{2D}{N} \left( \frac{2 \log( \lambda) + 7}{\sqrt{\lambda}} \right)
\end{eqnarray*}
Notice first that if we choose $\lambda<1$, this reduces to using only the Weil bound an therefore we get
\[
|A(m,\chi,N)| \leq \frac{2D}{N} \sum_{c=1}^{+ \infty} \frac{\tau (c)}{c^{3/2}} \leq \frac{14D}{N}
\]
We choose $\lambda = D/e^{7/2}$ and develop $\log(\lambda)$ in the expression above. We obtain the bound
\[
 |A(m,\chi,N)| \leq \frac{\sqrt{Dm}}{N} ( 9 \log(D)^2 + 6 \log(D) \log(N)).
\]
Actually, these are the dominant terms of the expression for chosen $\lambda$, and it is not hard to check that the remainder is negative.
If $D < e^{7/2}$, the bound also holds. Notice it improves the Weil bound only when $D$ is large enough ($D>1000$ is a good order of magnitude).
We obtain the Weil bound similarly for $B(m,\chi,N)$ and for $\lambda >D$, we have
\[
|B(m,\chi,N)| \leq \frac{6 \sqrt{Dm}}{\sqrt{N}} \left( (\log(D)+1.5)(1+ \log(\lambda)) + \frac{\log(\lambda)^2}{2} \right) + D \sqrt{m} \cdot \frac{2 \log(\lambda)+7}{\sqrt{\lambda}} + \frac{\tau(D) \sqrt{m}}{\sqrt{D}}.
\]
The last term in the sum comes from the fact we only have a Weil bound for $d=D$. 
Similarly, calculating the Weil-only bound and then choosing $\lambda= DN/e^{7/2}$, we obtain the desired result.
\end{proof}

Adding these bounds together, we have the following proposition.
\begin{prop}
Let $K$ be the imaginary quadratic field of discriminant $-D$ and Dirichlet character $\chi$.
For every prime number $p > 50 D^{1/4} \log(D)$ not dividing $D$, $(a_1,L_\chi)_{p^2} ^{+,new} \neq 0$.
\end{prop}

\begin{proof}
Thanks to formulas \eqref{formulea1chip2old} and \eqref{formuleamlchinamchibmchi}, we have
 \[
\begin{array}{cc}
\frac{|(a_1,L_\chi)_{p ^2} ^{+,new}|}{4 \pi} \geq (e^{-2 \pi / (D p)}- \frac{p}{p^2-1} - \frac{1}{p^2-1}) & - 2 \pi \left(|A(1,\chi,p^2| + \frac{p|A(1, \chi,p)|}{p^2-1} + \frac{|A(p,\chi,p)|}{p^2-1} \right) \\
& - 2 \pi \left(\frac{|B(1,\chi,p^2|}{p} + \frac{\sqrt{p}|B(1, \chi,p)|}{p^2-1}  + \frac{|B(p,\chi,p)|}{p^2-1} \right)
\end{array}
 \]
The first term is larger than 19/20 because $D \geq 3$ and $p \geq 50 \cdot 3^{1/4} \log(3) \geq 72$. Now, putting together the Abel-transform bounds from Proposition \ref{bornesamchiNabel}, we obtain 
\[
 \frac{|(a_1,L_\chi)_{p ^2} ^{+,new}|}{4 \pi} \geq \frac{19}{20} - \frac{\sqrt{D}}{p^2} \left(294 \log^2(D) + 416 \log(D) \log(p) + 227 \log^2 (p) \right) - \frac{2 \pi \tau(D)}{\sqrt{D}} \left( \frac{1}{p} + \frac{1}{p-1} \right)
\]
and after computation, we find that for $D \geq 15$, $p \geq 50 D^{1/4} \log(D)$, $(a_1,L_\chi)_{p ^2} ^{+,new} \neq 0$.
For $D < 15$, we have to recover this results with some more computation : when $D=7,8$ or $11$, this is possible by using Weil bounds for all six terms except for $|B(1,\chi,p^2)|$. When $D=3$ or $4$, we can obtain it back from computations we can find in \cite{BennettEllenbergNg}. More precisely, we mention this article provides sharp estimates for small discriminants ($p \geq 61$ for $D=4$ and $p \geq 97$ for $D=8$ in Lemma 14), hence we suggest using it in these cases rather than our own bound.
\end{proof}

 An equivalent of Akbary's trace formula for $(a_m,a_n)_N ^{w_M}$ with $M$ any divisor of $N$ such that $(M,N/M)=1$ (Akbary deals with $M=N$) could allow us to give estimates of $(a_1,L)_{dp^2}^{w_p,p-new}$ with prime $p$ and $d>1$ by the same methods, and thus  prove that the jacobian of 
$X_0(d) \times X_{nonsplit} (p)$ (described in \cite{DeSmitEdixhoven}) has a rank zero quotient when $p$ is large enough . This would extend a result of \cite{DarmonMerel} (where they consider the case $g(X_0(d))=0$) to some more cases.

\bibliographystyle{amsalpha}
\bibliography{bibliotdn}

\end{document}